\documentclass[11pt]{amsart}
\usepackage{amsfonts}
\usepackage{amsmath}
\usepackage{amssymb}
\usepackage{amsthm}
\usepackage{color, ulem}
\usepackage{tikz}
\tikzstyle{vertex}=[circle, draw, fill, inner sep=0pt, minimum size=5pt]
\tikzstyle{circ}=[circle, draw, inner sep=0pt, minimum size=15pt]
\tikzstyle{leg}=[inner sep=0pt, minimum size=0pt]
\tikzstyle{loop}=[looseness=15]
\tikzstyle{loopcirc}=[looseness=5]

\input xy
\xyoption {all}

\def \AA {{\mathsf{A}}^\star}

\def \NL {{\mathsf{NL}}^\star}

\def \MM {{\overline{\mathsf{M}}}}

\def \proj {{\mathbb{P}}}
\def \com {{\mathbb{C}}}
\def \Z {{\mathbb{Z}}}

\def \w {\widehat}
\def \ww {\widetilde}

\newcommand{\comment}[1]{}

\newtheorem{theorem}{Theorem}

\newtheorem{conjecture}{Conjecture}

\newtheorem{corollary}[theorem]{Corollary}
\newtheorem{proposition}{Proposition}
\theoremstyle{definition}
 
\theoremstyle{definition}

\begin{document}
\baselineskip=16pt

\title[Relations in the tautological ring]
{Relations in the tautological ring of the
moduli space of $K3$ surfaces}

\author{Rahul Pandharipande}
\address{Department of Mathematics, ETH Z\"urich}
\email {rahul@math.ethz.ch}
\author{Qizheng Yin}
\address{Department of Mathematics, ETH Z\"urich}
\email {qizheng.yin@math.ethz.ch}
\date{July 2016}

\begin{abstract} 
We study the interplay of the moduli of curves and the
moduli of $K3$ surfaces via the virtual class of
the moduli spaces of stable maps. Using 
Getzler's relation in genus 1, we construct a
universal decomposition of the diagonal in Chow in the
third fiber product of the universal $K3$ surface.
The decomposition has terms supported on Noether-Lefschetz
loci which are not visible in the Beauville-Voisin
decomposition for a fixed $K3$ surface.
As a result of our universal decomposition, we prove
the conjecture of Marian-Oprea-Pandharipande: the full
tautological ring of the moduli space of $K3$ surfaces
is generated in Chow by the classes of the Noether-Lefschetz
loci. Explicit boundary relations are constructed for all
$\kappa$ classes.

More generally, we propose a connection between relations
in the tautological ring of the moduli spaces of curves
and relations in the tautological ring
of the moduli space of $K3$ surfaces.
The WDVV relation in genus 0 is used in our proof 
of the MOP conjecture.
\end{abstract}

\maketitle

\setcounter{tocdepth}{1} 
\tableofcontents

\setcounter{section}{-1}
\newpage

\section{Introduction}

\subsection{$\kappa$ classes} \label{MMM}

Let $\mathcal{M}_{2\ell}$ be the moduli space of quasi-polarized
$K3$ surfaces $(X,H)$ of degree $2\ell>0$: 
\begin{enumerate}
\item[$\bullet$]  $X$ is a nonsingular, projective $K3$ surface over $\com$, 
\item[$\bullet$] 
$H\in \text{Pic}(X)$ is a primitive and nef class
satisfying
$$\langle H, H\rangle_X\,  =\,  \int_X H^2\,  =\,  2\ell\, .$$ 
\end{enumerate}
The basics of quasi-polarized $K3$ surfaces and their
moduli are reviewed in Section \ref{ooo}.

Consider the universal quasi-polarized $K3$ surface  over the moduli space,
$$ \pi: \mathcal{X} \rightarrow \mathcal{M}_{2\ell}\, .$$
We define a {canonical} divisor class on the universal surface,
$$\mathcal{H} \ \in \ \mathsf{A}^1(\mathcal{X},\mathbb{Q})\, ,$$
which restricts to $H$ on the fibers of $\pi$ 
by the following construction.
Let $\MM_{0,1}(\pi,H)$ be the $\pi$-relative moduli space of stable
maps: $\MM_{0,1}(\pi,H)$ parameterizes stable maps from genus 0 curves
with 1 marked point to the fibers of $\pi$ representing the fiberwise class
$H$. Let 
$$\epsilon: \MM_{0,1}(\pi,H) \rightarrow \mathcal{X}\, $$
be the evaluation morphism over $\mathcal{M}_{2\ell}$. 
The moduli space $\MM_{0,1}(\pi,H)$ carries a $\pi$-relative
reduced obstruction theory with reduced virtual class of $\pi$-relative
dimension $1$. We define
$$\mathcal{H} \, = \, \frac{1}{N_0(\ell)} \,\cdot \, \epsilon_*\left[ \MM_{0,1}(\pi,H)\right]^{\text{red}} 
\ \in \ \mathsf{A}^1(\mathcal{X},\mathbb{Q}) \, ,$$
where $N_0(\ell)$ is the genus 0 Gromov-Witten invariant{\footnote{While
$\ell>0$ is required for the quasi-polarization $(X,H)$, the reduced
Gromov-Witten invariant~$N_0(\ell)$ is well-defined for all $\ell\geq -1$.}}
$$N_0(\ell) = \int_{[\MM_{0,0}(X,H)]^{\text{red}}} 1\,. 
$$
By the Yau-Zaslow formula{\footnote{The formula was proposed
in \cite{yauz}. The first
proofs in the primitive case can be found in \cite{bea,brl}.
We will later require the 
full Yau-Zaslow formula for the genus 0 Gromov-Witten counts also in imprimitive classes
proven in \cite{KMPS}.}}, the invariant $N_0(\ell)$ is never 0 for $\ell\geq -1$,
$$\sum_{\ell=-1}^\infty q^\ell N_0(\ell) \,=\, \frac{1}{q}+ 24 + 324 q + 3200 q^2 \ldots\, .$$
The construction of $\mathcal{H}$ is discussed further in Section \ref{fafa2}.

The $\pi$-relative tangent bundle of $\mathcal{X}$,  
$$
\mathcal{T}_\pi\rightarrow \mathcal{X} \, ,
$$
is of rank 2 and is canonically defined.
Using $\mathcal{H}$ and $c_2(\mathcal{T}_\pi)$, we define
the $\kappa$ classes,
$$\kappa_{[a;b]} \, = \, \pi_*\left(\mathcal{H}^a \cdot 
c_2(\mathcal{T}_\pi)^b\right) 
\ \in \mathsf{A}^{a+2b-2}(\mathcal{M}_{2\ell},\mathbb{Q})\, .$$
Our definition follows \cite[Section 4]{MOP3} {\it except for the
canonical choice of $\mathcal{H}$}. The construction
here requires {\it no} choices to be made in the
definition of the $\kappa$ classes.

\subsection{Strict tautological classes} \label{hah2}

The Noether-Lefschetz loci
also define
classes in the Chow ring $\AA(\mathcal{M}_{2\ell},\mathbb{Q})$. Let
$$\NL(\mathcal{M}_{2\ell}) \subset \AA(\mathcal{M}_{2\ell},\mathbb{Q})$$
be the subalgebra generated by the Noether-Lefschetz loci (of all codimensions).
On the Noether-Lefschetz locus{\footnote{We view the Noether-Lefschetz
loci as proper maps to $\mathcal{M}_{2\ell}$ instead of subspaces.}} 
$$\mathcal{M}_{\Lambda}\rightarrow \mathcal{M}_{2\ell}\, ,$$
 corresponding to 
the larger Picard lattice $\Lambda \supset (2\ell)$, richer $\kappa$
classes may be defined by simultaneously using several elements of
$\Lambda$. 

We define {\it canonical} $\kappa$ classes based on the lattice polarization $\Lambda$.
A nonzero class
$L\in \Lambda$ is {\it admissible} if
\begin{enumerate}
\item[(i)] $L= m \cdot \widetilde{L}$ with $\widetilde{L}$ primitive,
$m>0$, and 
$\langle \widetilde{L}, \widetilde{L}\rangle_\Lambda\geq -2$,
\item[(ii)] $\langle H, L\rangle_\Lambda \geq 0$,
\end{enumerate}
and in case of equality in (ii), which forces equality in (i) by the Hodge index theorem,
\begin{enumerate}
\item[(ii')] $L$ is effective.
\end{enumerate}
Effectivity is
{\it equivalent} to the condition
$$\langle H, L\rangle_\Lambda \geq 0\, $$
for {\it every} quasi-polarization $H\in \Lambda$ for a generic
$K3$ surface parameterized by $\mathcal{M}_\Lambda$.


For $L\in \Lambda$ admissible, we define
$$\mathcal{L} \, = \, \frac{1}{N_0(L)} \,\cdot \, 
\epsilon_*\left[ \MM_{0,1}(\pi_\Lambda,L)\right]^{\text{red}} 
\ \in \ \mathsf{A}^1(\mathcal{X}_\Lambda,\mathbb{Q}) \, ,$$
where $\pi_\Lambda:\mathcal{X}_\Lambda \rightarrow \mathcal{M}_{\Lambda}$ is
the universal $K3$ surface.
The reduced Gromov-Witten invariant
$$N_0(L) = \int_{[\MM_{0,0}(X,L)]^{\text{red}}} 1$$
is nonzero for all admissible classes by the full Yau-Zaslow
formula proven in \cite{KMPS}, see Section \ref{gen0}.

For $L_1,\ldots,L_k\in \Lambda$  admissible classes, we have 
canonically constructed divisors
$$\mathcal{L}_1,\ldots,\mathcal{L}_k 
\ \in \ \mathsf{A}^1(\mathcal{X}_\Lambda,\mathbb{Q})\, .$$
We define the richer $\kappa$ classes on $\mathcal{M}_\Lambda$ by
\begin{equation}\label{xrrx}
\kappa_{[L_1^{a_1},\ldots,L_k^{a_k};b]} \, = \,
\pi_{\Lambda*}\left( \mathcal{L}_1^{a_1}\cdots
\mathcal{L}_k^{a_k}
 \cdot 
c_2(\mathcal{T}_{\pi_\Lambda})^b\right) 
\ \in \ \mathsf{A}^{\sum_i a_i+2b-2}(\mathcal{M}_{\Lambda},\mathbb{Q})\, .
\end{equation}
We will sometimes suppress the dependence on the $L_i$,
$$\kappa_{[L_1^{a_1},\ldots,L_k^{a_k}; b]}=
\kappa_{[{a_1},\ldots,{a_k}; b]}\, .$$

We define the {\it strict tautological ring} of the moduli space
of $K3$ surfaces,
$${\mathsf{R}}^\star(\mathcal{M}_{2\ell}) \subset \AA(\mathcal{M}_{2\ell},\mathbb{Q})\, ,$$
to be the subring generated by the push-forwards from
the Noether-Lefschetz loci $\mathcal{M}_\Lambda$ of 
all products of
the $\kappa$ classes \eqref{xrrx} obtained
from admissible classes of $\Lambda$. By definition, 
$$\NL(\mathcal{M}_{2\ell}) \subset {\mathsf{R}}^\star(\mathcal{M}_{2\ell})\, .$$
There is no need to include a $\kappa$ index for the
first Chern class of $\mathcal{T}_\pi$ since
$$c_1(\mathcal{T}_\pi) = -\pi^*\lambda$$ where $\lambda=c_1(\mathbb E)$ is the first Chern class of the Hodge line bundle
$$\mathbb{E} \rightarrow \mathcal{M}_{2\ell}$$
with fiber $H^0(X,K_X)$ over the moduli
point $(X,H)\in \mathcal{M}_{2\ell}$. 
The Hodge class $\lambda$ is known to be supported on Noether-Lefschetz 
divisors.{\footnote{
By \cite[Theorem 1.2]{BKPSB},  $\lambda$ on $\mathcal{M}_\Lambda$ is supported on Noether-Lefschetz divisors for every lattice polarization~$\Lambda$. See
also \cite[Theorem 3.1]{DM} for a stronger statement: $\lambda$ on $\mathcal{M}_{2\ell}$ is supported on any infinite collection of Noether-Lefschetz divisors.}}

A slightly different {tautological ring} of the moduli space
of $K3$ surfaces
was defined in \cite{MOP3}. 
A basic result conjectured in \cite{MP} and proven in \cite{Ber} is the
isomorphism
$$\mathsf{NL}^1(\mathcal{M}_{2\ell}) = \mathsf{A}^1(\mathcal{M}_{2\ell},\mathbb{Q})\, .$$
In fact, the Picard group of $\mathcal{M}_{\Lambda}$ is generated
by the Noether-Lefschetz divisors of $\mathcal{M}_{\Lambda}$ 
for every lattice polarization $\Lambda$ of rank $\leq 17$ by \cite{Ber}.
As an immediate consequence, the strict tautological ring
defined here is isomorphic to the tautological ring of \cite{MOP3}
in all codimensions up to 17. Since the dimension of
$\mathcal{M}_{2\ell}$ is 19, the differences in the two definitions
are only possible in degrees 18 and 19.

We prefer to work with the strict tautological ring. A basic
advantage is that the $\kappa$ classes
are defined canonically (and not {\it up to twist} as in \cite{MOP3}).
Every class of the strict tautological ring $\mathsf{R}^\star(\mathcal{M}_{2\ell})$ is defined 
explicitly. A central result of the paper is the following generation property conjectured
first in \cite{MOP3}.

\begin{theorem} \label{dxxd} The strict tautological ring is generated
by Noether-Lefschetz loci,
$$\NL(\mathcal{M}_{2\ell}) = \mathsf{R}^\star(\mathcal{M}_{2\ell})\, .$$
\end{theorem}

Our construction also defines 
the strict tautological ring 
$${\mathsf{R}}^\star(\mathcal{M}_{\Lambda})\subset
\mathsf{A}^\star(\mathcal{M}_{\Lambda},\mathbb{Q})$$
for every lattice polarization $\Lambda$.
As before, the subring generated by the Noether-Lefschetz
loci corresponding to lattices $\widetilde{\Lambda} \supset
\Lambda$ is contained in the strict tautological~ring,
$$\NL(\mathcal{M}_{{\Lambda}}) \subset {\mathsf{R}}^\star(\mathcal{M}_{\Lambda})\, .$$
In fact, we prove a generation result parallel to Theorem
\ref{dxxd} for every lattice
polarization,
$$\NL(\mathcal{M}_{\Lambda}) = \mathsf{R}^\star(\mathcal{M}_{\Lambda}) \,.$$


\subsection{Fiber products of the universal surface} \label{ttun}
Let $\mathcal{X}^n$ denote the $n^{\text{th}}$ fiber product of the
universal $K3$ surface over $\mathcal{M}_{2\ell}$,
$$\pi^n: \mathcal{X}^n \rightarrow \mathcal{M}_{2\ell}\, .$$
The strict tautological ring 
$${\mathsf{R}}^\star(\mathcal{X}^n) \subset 
\mathsf{A}^\star(\mathcal{X}^n,\mathbb{Q})$$
is defined 
to be the subring generated 
by the push-forwards to $\mathcal{X}^n$ from the Noether-Lefschetz loci
$$\pi^n_\Lambda: \mathcal{X}^n_\Lambda \rightarrow \mathcal{M}_{\Lambda}\, $$
of all products of 
\begin{enumerate}
\item[$\bullet$] the $\pi^n_\Lambda$-relative diagonals in $\mathcal{X}_\Lambda^n$,
\item[$\bullet$] the pull-backs of $\mathcal{L} \in \mathsf{A}^1(\mathcal{X}_\Lambda,\mathbb{Q})$ via the $n$ projections
$$\mathcal{X}_\Lambda^n \rightarrow \mathcal{X}_\Lambda$$
for every admissible $L\in \Lambda$,
\item[$\bullet$] the pull-backs of $c_2(\mathcal{T}_{\pi_\Lambda}) \in \mathsf{A}^2(\mathcal{X}_\Lambda,\mathbb{Q})$ via the $n$ projections,
\item[$\bullet$] the pull-backs of ${\mathsf{R}}^\star({\mathcal{M}}_{\Lambda})$
via $\pi^{n*}_\Lambda$.
\end{enumerate}
The construction also defines 
the strict tautological ring 
$${\mathsf{R}}^\star(\mathcal{X}^n_{\Lambda})\subset
\mathsf{A}^\star(\mathcal{X}^n_{\Lambda},\mathbb{Q})$$
for every lattice polarization $\Lambda$.

\subsection{Export construction} \label{conjs}
Let $\MM_{g,n}(\pi_\Lambda,L)$
be the $\pi_\Lambda$-relative moduli space of stable maps representing the  admissible class $L\in \Lambda$. The 
evaluation map at the $n$ markings is
$$\epsilon^n: \MM_{g,n}(\pi_\Lambda,L) \rightarrow \mathcal{X}^n_\Lambda\, .$$

\begin{conjecture} \label{conj1}
The push-forward of the reduced
virtual fundamental class lies in the strict tautological ring,
$$\epsilon_*^n \left[ \MM_{g,n}(\pi_\Lambda,L)\right]^{\textup{red}} 
\ \in \ {\mathsf{R}}^\star(\mathcal{X}^n_\Lambda)\, .$$
\end{conjecture}

When Conjecture \ref{conj1} is restricted to a fixed $K3$ surface $X$, 
another open question is obtained.

\begin{conjecture} \label{conj2}
The push-forward of the reduced
virtual fundamental class,
$$\epsilon_*^n \left[ \MM_{g,n}(X,L)\right]^{\textup{red}} 
\ \in \ {\mathsf{A}}^\star(X^n,\mathbb{Q})\, ,$$
lies in the Beauville-Voisin ring of $X^n$ generated by the diagonals and
the pull-backs of $\textup{Pic}(X)$ via the $n$ projections.
\end{conjecture}

If Conjecture \ref{conj1} could be proven also for descendents (and in an effective form), 
then we could export
tautological relations on $\MM_{g,n}$ to $\mathcal{X}^n_\Lambda$ via the
morphisms
$$ \MM_{g,n}\ \stackrel{\tau}{\longleftarrow} \ \MM_{g,n}(\pi_\Lambda, L) \ 
\stackrel{\epsilon_\Lambda^n}{\longrightarrow}\  
\mathcal{X}^n_{\Lambda}\, .$$
More precisely, given a relation $\mathsf{Rel}$ among tautological
classes on $\MM_{g,n}$,
$$\epsilon_*^n \tau^*(\mathsf{Rel}) \,=\, 0 \ \in\ {\mathsf{R}}^\star
(\mathcal{X}^n_\Lambda)$$
would then be a relation among strict tautological classes on
$\mathcal{X}^n_\Lambda$.

We prove Theorem \ref{dxxd} as a consequence of the export construction
for the WDVV relation in genus 0 and for Getzler's relation in genus 1.
The required parts of Conjectures~\ref{conj1} and \ref{conj2} are proven by hand.

\subsection{WDVV and Getzler}
We fix an admissible class $L \in \Lambda$ and the corresponding 
divisor ${\mathcal{L}} \in \mathsf{A}^1(\mathcal{X}_{\Lambda}, \mathbb{Q})$. For $i \in \{1, \ldots, n\}$, let
$${\mathcal{L}}_{(i)} \ \in \ \mathsf{A}^1(\mathcal{X}_\Lambda^n, \mathbb{Q})$$
denote 
the pull-back of ${\mathcal{L}}$ via the $i^{\text{th}}$ projection
$$\text{pr}_{(i)} : \mathcal{X}_\Lambda^n \to \mathcal{X}_\Lambda \,.$$
For $1 \leq i < j \leq n$, let
$$\Delta_{(ij)} \ \in \ \mathsf{A}^2(\mathcal{X}_\Lambda^n, \mathbb{Q})$$
be the $\pi_\Lambda^n$-relative diagonal where the $i^{\text{th}}$ and $j^{\text{th}}$ coordinates are
equal. We  write $$\Delta_{(ijk)} \, = \, \Delta_{(ij)} \cdot \Delta_{(jk)} \ \in \ \mathsf{A}^4(\mathcal{X}_\Lambda^n, \mathbb{Q}) \,.$$

The Witten-Dijkgraaf-Verlinde-Verlinde relation in genus 0 is
\begin{equation} \label{wdvv}
\left[\begin{tikzpicture}[baseline={([yshift=-.5ex]current bounding box.center)}]
	\node[leg] (l3) at (0,2) [label=above:$3$] {};
	\node[leg] (l4) at (1,2) [label=above:$4$] {};
	\node[vertex] (v2) at (.5,1.5) [label=right:$0$] {};
	\node[vertex] (v1) at (.5,.5) [label=right:$0$] {};
	\node[leg] (l1) at (0,0) [label=below:$1$] {};
	\node[leg] (l2) at (1,0) [label=below:$2$] {};
	\path
		(l3) edge (v2)
		(l4) edge (v2)
		(v2) edge (v1)
		(v1) edge (l1)
		(v1) edge (l2)
	;
\end{tikzpicture}\right] \ - \ \left[\begin{tikzpicture}[baseline={([yshift=-.5ex]current bounding box.center)}]
	\node[leg] (l2) at (0,2) [label=above:$2$] {};
	\node[leg] (l4) at (1,2) [label=above:$4$] {};
	\node[vertex] (v2) at (.5,1.5) [label=right:$0$] {};
	\node[vertex] (v1) at (.5,.5) [label=right:$0$] {};
	\node[leg] (l1) at (0,0) [label=below:$1$] {};
	\node[leg] (l3) at (1,0) [label=below:$3$] {};
	\path
		(l2) edge (v2)
		(l4) edge (v2)
		(v2) edge (v1)
		(v1) edge (l1)
		(v1) edge (l3)
	;
\end{tikzpicture}\right] \ = \ 0 \ \in \ \mathsf{A}^1(\MM_{0, 4}, \mathbb{Q})\, .
\end{equation}

\vspace{0pt}
\begin{theorem} \label{WDVV}
For all admissible $L\in \Lambda$,
exportation of the  WDVV relation yields
\begin{multline} \tag{\dag}
{\mathcal{L}}_{(1)} {\mathcal{L}}_{(2)} {\mathcal{L}}_{(3)} \Delta_{(34)} + {\mathcal{L}}_{(1)}{\mathcal{L}}_{(3)} {\mathcal{L}}_{(4)} \Delta_{(12)} \\ 
- {\mathcal{L}}_{(1)}{\mathcal{L}}_{(2)}{\mathcal{L}}_{(3)}\Delta_{(24)} - {\mathcal{L}}_{(1)} {\mathcal{L}}_{(2)}{\mathcal{L}}_{(4)}\Delta_{(13)}
+ \ldots 
\, = \, 0 \ \in \ \mathsf{A}^5(\mathcal{X}_\Lambda^4, \mathbb{Q})\,,
\end{multline}
where the dots stand for strict tautological classes supported over proper Noether-Lefschetz divisors
of $\mathcal{M}_\Lambda$.
\end{theorem}


Getzler \cite{getz} in 1997 discovered a beautiful relation  in the
cohomology of $\MM_{1,4}$ which was proven to hold in Chow in \cite{pan}:
\begin{multline} \label{getzler}
12\left[\ \begin{tikzpicture}[baseline={([yshift=-.5ex]current bounding box.center)}]
	\node[leg] (l3) at (0,3) {};
	\node[leg] (l4) at (1,3) {};
	\node[vertex] (v3) at (.5,2.5) [label=right:$0$] {};
	\node[vertex] (v2) at (.5,1.5) [label=right:$1$] {};
	\node[vertex] (v1) at (.5,.5) [label=right:$0$] {};
	\node[leg] (l1) at (0,0) {};
	\node[leg] (l2) at (1,0) {};
	\path
		(l3) edge (v3)
		(l4) edge (v3)
		(v3) edge (v2)
		(v2) edge (v1)
		(v1) edge (l1)
		(v1) edge (l2)
	;
\end{tikzpicture}\ \right] \ - \ 4\left[\ \begin{tikzpicture}[baseline={([yshift=-.5ex]current bounding box.center)}]
	\node[leg] (l3) at (0,3) {};
	\node[leg] (l4) at (1,3) {};
	\node[vertex] (v3) at (.5,2.5) [label=right:$0$] {};
	\node[leg] (l2) at (0,1.5) {};
	\node[vertex] (v2) at (.5,1.5) [label=right:$0$] {};
	\node[vertex] (v1) at (.5,.5) [label=right:$1$] {};
	\node[leg] (l1) at (0,0) {};
	\path
		(l3) edge (v3)
		(l4) edge (v3)
		(v3) edge (v2)
		(l2) edge (v2)
		(v2) edge (v1)
		(v1) edge (l1)
	;
\end{tikzpicture}\ \right] \ - \ 2\left[\ \begin{tikzpicture}[baseline={([yshift=-.3ex]current bounding box.center)}]
	\node[leg] (l3) at (0,2.5) {};
	\node[leg] (l4) at (1,2.5) {};
	\node[vertex] (v3) at (.5,2) [label=right:$0$] {};
	\node[leg] (l2) at (0,1.5) {};
	\node[leg] (l1) at (0,.5) {};
	\node[vertex] (v2) at (.5,1) [label=right:$0$] {};
	\node[vertex] (v1) at (.5,0) [label=right:$1$] {};
	\path
		(l3) edge (v3)
		(l4) edge (v3)
		(v3) edge (v2)
		(l2) edge (v2)
		(l1) edge (v2)
		(v2) edge (v1)
	;
\end{tikzpicture}\ \right] \ + \ 6\left[\ \begin{tikzpicture}[baseline={([yshift=-.3ex]current bounding box.center)}]
	\node[leg] (l2) at (0,2.5) {};
	\node[leg] (l3) at (.5,2.5) {};
	\node[leg] (l4) at (1,2.5) {};
	\node[vertex] (v3) at (.5,2) [label=right:$0$] {};
	\node[leg] (l1) at (0,1) {};
	\node[vertex] (v2) at (.5,1) [label=right:$0$] {};
	\node[vertex] (v1) at (.5,0) [label=right:$1$] {};
	\path
	    (l2) edge (v3)
		(l3) edge (v3)
		(l4) edge (v3)
		(v3) edge (v2)
		(l1) edge (v2)
		(v2) edge (v1)
	;
\end{tikzpicture}\ \right] \\[3pt]
+ \ \left[\begin{tikzpicture}[baseline={([yshift=-.3ex]current bounding box.center)}]
	\node[leg] (l2) at (0,1.5) {};
	\node[leg] (l3) at (.5,1.5) {};
	\node[leg] (l4) at (1,1.5) {};
	\node[vertex] (v2) at (.5,1) [label=right:$0$] {};
	\node[leg] (l1) at (0,.5) {};
	\node[vertex] (v1) at (.5,0) [label=right:$0$] {};
	\path
        (l2) edge (v2)
		(l3) edge (v2)
		(l4) edge (v2)
		(v2) edge (v1)
		(l1) edge (v1)
		(v1) edge[in=-135,out=-45,loop] (v1)
	;
\end{tikzpicture}\right] \ + \ \left[\begin{tikzpicture}[baseline={([yshift=-.3ex]current bounding box.center)}]
	\node[leg] (l1) at (0,1.5) {};
	\node[leg] (l2) at (.33,1.5) {};
	\node[leg] (l3) at (.67,1.5) {};
	\node[leg] (l4) at (1,1.5) {};
	\node[vertex] (v2) at (.5,1) [label=right:$0$] {};
	\node[vertex] (v1) at (.5,0) [label=right:$0$] {};
	\path
        (l1) edge (v2)
        (l2) edge (v2)
		(l3) edge (v2)
		(l4) edge (v2)
		(v2) edge (v1)
		(v1) edge[in=-135,out=-45,loop] (v1)
	;
\end{tikzpicture}\right] \ - \ 2\left[\ \begin{tikzpicture}[baseline={([yshift=-.5ex]current bounding box.center)}]
	\node[leg] (l3) at (0,2) {};
	\node[leg] (l4) at (1,2) {};
	\node[vertex] (v2) at (.5,1.5) [label=right:$0$] {};
	\node[vertex] (v1) at (.5,.5) [label=right:$0$] {};
	\node[leg] (l1) at (0,0) {};
	\node[leg] (l2) at (1,0) {};
	\path
		(l3) edge (v2)
		(l4) edge (v2)
		(v2) edge[bend left=60] (v1)
		(v2) edge[bend right=60] (v1)
		(v1) edge (l1)
		(v1) edge (l2)
	;
\end{tikzpicture}\ \right] \ = \ 0 \ \in \ \mathsf{A}^2(\MM_{1, 4}, \mathbb{Q})\, .
\end{multline}

\vspace{8pt}
\noindent Here, the strata are summed over all marking
distributions and are taken in the stack sense (following the conventions of \cite{getz}). 

\begin{theorem}\label{ggg} For admissible $L\in \Lambda$
satisfying the condition $\langle L, L\rangle_\Lambda \geq 0$,
exportation of Getzler's relation yields
\begin{multline} \tag{\ddag}
{\mathcal{L}}_{(1)}\Delta_{(12)}\Delta_{(34)} + {\mathcal{L}}_{(3)}\Delta_{(12)}\Delta_{(34)} + {\mathcal{L}}_{(1)}\Delta_{(13)}\Delta_{(24)} 
+ {\mathcal{L}}_{(2)}\Delta_{(13)}\Delta_{(24)} + {\mathcal{L}}_{(1)}\Delta_{(14)}\Delta_{(23)} \\ + {\mathcal{L}}_{(2)}\Delta_{(14)}\Delta_{(23)} 
- {\mathcal{L}}_{(1)}\Delta_{(234)} - {\mathcal{L}}_{(2)}\Delta_{(134)} - {\mathcal{L}}_{(3)}\Delta_{(124)} - {\mathcal{L}}_{(4)}\Delta_{(123)} \\
- {\mathcal{L}}_{(1)}\Delta_{(123)} - {\mathcal{L}}_{(1)}\Delta_{(124)}  - {\mathcal{L}}_{(1)}\Delta_{(134)} - {\mathcal{L}}_{(2)}\Delta_{(234)}
+\ldots
\, = \, 0 \ \in \ \mathsf{A}^5(\mathcal{X}_\Lambda^4, \mathbb{Q})\,, 
\end{multline}
where the dots stand for strict tautological classes supported over proper Noether-Lefschetz loci
of $\mathcal{M}_\Lambda$.
\end{theorem}

The statements of Theorems \ref{WDVV} and \ref{ggg} contain only the {\it principal} terms of the
relation (not supported
over proper Noether-Lefschetz loci
of $\mathcal{M}_\Lambda$). We will write all the terms represented by the dots in Sections \ref{wwww} and \ref{gggg}.

The relation of Theorem \ref{WDVV} is obtained from the
export construction after 
dividing by the genus 0 reduced Gromov-Witten invariant
$N_0(L)$. The latter never vanishes for admissible classes.
Similarly, for Theorem \ref{ggg}, the export construction 
has been divided by the genus 1
reduced Gromov-Witten invariant
$$N_1(L) = \int_{[\MM_{1, 1}(X, L)]^{\text{red}}} \text{ev}^*(\mathsf{p}) \,,$$
where $\mathsf{p} \in H^4(X, \mathbb{Q})$ is the class of a point on $X$. 
By a result of Oberdieck discussed in Section \ref{gen1}, 
$N_1(L)$ does not vanish for admissible classes satisfying
$\langle L, L\rangle_\Lambda \geq 0$.


\subsection{Relations on $\mathcal{X}^3_\Lambda$}

As a Corollary of Getzler's relation, we have the following result. Let
$$\text{pr}_{(123)} : \mathcal{X}_\Lambda^4 \to \mathcal{X}_\Lambda^3$$
be the projection to the first 3 factors. 
Let $L=H$ and
consider the operation
$$\text{pr}_{(123)*} (\mathcal{H}_{(4)} \cdot -)$$
applied to the relation ($\ddag$). 
We obtain a universal decomposition of the diagonal $\Delta_{(123)}$ which generalizes the result of Beauville-Voisin \cite{BV} for a fixed $K3$ surface.

\begin{corollary} \label{bvdiag}
The $\pi^3_\Lambda$-relative diagonal $\Delta_{(123)}$ admits a decomposition with
principal terms
\begin{multline} \tag{$\ddag'$}
2\ell \cdot \Delta_{(123)} \, = \, \mathcal{H}_{(1)}^2 \Delta_{(23)} + \mathcal{H}_{(2)}^2 \Delta_{(13)} + \mathcal{H}_{(3)}^2 \Delta_{(12)} \\
- \mathcal{H}_{(1)}^2 \Delta_{(12)} - \mathcal{H}_{(1)}^2 \Delta_{(13)} - \mathcal{H}_{(2)}^2 \Delta_{(23)} 
+\ldots
\ \in \ \mathsf{A}^4(\mathcal{X}_\Lambda^3, \mathbb{Q})\,,
\end{multline}
where the dots stand for strict tautological classes supported over proper Noether-Lefschetz loci
of $\mathcal{M}_\Lambda$.
\end{corollary}

The diagonal $\Delta_{(123)}$ controls the behavior of the $\kappa$ classes. For instance, we have
$$\kappa_{[a;b]} \, = \, \pi^3_*\left(\mathcal{H}_{(1)}^a \cdot \Delta_{(23)}^b \cdot \Delta_{(123)}\right) \ \in \ \mathsf{A}^{a + 2b - 2}(\mathcal{M}_{2\ell}, \mathbb{Q})\,.$$
The diagonal decomposition of Corollary \ref{bvdiag} plays a fundamental role in the proof of Theorem \ref{dxxd}.

\subsection{Cohomological results}
Bergeron and Li have an announced an
independent approach to the generation 
(in most codimensions)
of the
tautological ring $\mathsf{RH}^\star(\mathcal{M}_\Lambda)$ by 
Noether-Lefschetz loci in cohomology.
Petersen \cite{Pet} has proven the vanishing{\footnote{We use the complex grading
here.}}
$$\mathsf{RH}^{18}(\mathcal{M}_{2\ell})=
\mathsf{RH}^{19}(\mathcal{M}_{2\ell})=0\, .$$
We expect the above vanishing to hold also in Chow.

What happens in codimension 17 is a
very interesting question. By a result
of van der Geer and Katsura \cite{kvg},
$$\mathsf{RH}^{17}(\mathcal{M}_{2\ell})\neq 0\, .$$
We hope the stronger statement
\begin{equation} \label{hope5}
\mathsf{RH}^{17}(\mathcal{M}_{2\ell}) =\mathbb{Q}
\end{equation}
holds. If true, \eqref{hope5} would open
the door to a numerical theory of
proportionalities in the tautological ring.
The evidence for \eqref{hope5} is
rather limited at the moment. Careful
calculations in the $\ell=1$ and $2$
cases would be very helpful here. 

\subsection{Acknowledgments} 
We are grateful to  G.~Farkas, G.~van der Geer, D.~Huybrechts, Z.~Li, A.~Marian, D.~Maulik, G.~Oberdieck, D.~Oprea, D.~Petersen, and
J.~Shen for many discussions about
the moduli of $K3$ surfaces.
The paper was completed at the conference {\it Curves on
surfaces and threefolds} at the Bernoulli center in Lausanne 
in June~2016 attended
by both authors.

R.~P. was partially supported by 
SNF-200021143\-274, SNF-200020162928, ERC-2012-AdG-320368-MCSK, SwissMAP, and
the Einstein Stiftung. 
Q.~Y. was supported by the grant ERC-2012-AdG-320368-MCSK.

\section{$K3$ surfaces} \label{ooo}

\subsection{Reduced Gromov-Witten theory} \label{yzc}
Let $X$ be a nonsingular, projective $K3$ surface over $\com$, and let 
$$L \ \in \ \text{Pic}(X) \, =\, H^2(X,\mathbb{Z}) \cap H^{1,1}(X,\com)$$
be a nonzero effective class.
The moduli space ${\MM}_{g,n}(X,L)$
of genus $g$ stable maps with $n$ marked points
has expected dimension
$$\text{dim}^{\text{vir}}_\com\ {\MM}_{g,n}(X,\beta) 
= \int_L c_1(X) + (\text{dim}_\com(X) -3)(1-g) +n = g-1+n\,.$$
However, as the obstruction theory admits a 1-dimensional trivial quotient,
the virtual class $[{\MM}_{g,n}(X,L)]^{\text{vir}}$ vanishes.
The standard Gromov-Witten theory is trivial.

Curve counting on $K3$ surfaces
is captured instead by the {\it reduced} Gromov-Witten
theory constructed first via the twistor family in \cite{brl}.
An algebraic construction following~\cite{BF} is given in 
\cite{MP}. The reduced class 
$$\left[{\MM}_{g,n}(X,L)\right]^{\text{red}} \ \in \ \mathsf{A}_{g+n}({\MM}_{g,n}(X,L), \mathbb{Q})$$
has dimension $g+n$. 
The reduced Gromov-Witten integrals of $X$,
\begin{equation}\label{veq}
\Big\langle \tau_{a_1}(\gamma_1) \cdots \tau_{a_n}(\gamma_n) 
\Big\rangle_{g,L}^{X,\text{red}} \, = \, \int_{[{\MM}_{g,n}(X,L)]^{\text{red}}} \prod_{i=1}^n \text{ev}_i^*(\gamma_i)\cup \psi_i^{a_i}
\ \in \ \mathbb{Q}\,,
\end{equation}
are well-defined. 
Here, $\gamma_i \in H^\star(X,\mathbb{Q})$ and $\psi_i$ is the standard
descendent class at the $i^{\text{th}}$ marking.
Under deformations of $X$ for which $L$ remains a $(1,1)$-class,
the integrals \eqref{veq} are invariant.

\subsection{Curve classes on $K3$ surfaces}

Let $X$ be a nonsingular, projective $K3$ surface over $\mathbb{C}$.
The second cohomology of $X$ is a rank 22 lattice
with intersection form 
\begin{equation}\label{ccet}
H^2(X,\mathbb{Z}) \cong U\oplus U \oplus U \oplus E_8(-1) \oplus E_8(-1)\,,
\end{equation}
where
$$U
= \left( \begin{array}{cc}
0 & 1 \\
1 & 0 \end{array} \right)$$
and 
$$
E_8(-1)=  \left( \begin{array}{cccccccc}
 -2&    0 &  1 &   0 &   0 &   0 &   0 & 0\\
    0 &   -2 &   0 &  1 &   0 &   0 &   0 & 0\\
     1 &   0 &   -2 &  1 &   0 &   0 & 0 &  0\\
      0  & 1 &  1 &   -2 &  1 &   0 & 0 & 0\\
      0 &   0 &   0 &  1 &   -2 &  1 & 0&  0\\
      0 &   0&    0 &   0 &  1 &  -2 &  1 & 0\\ 
      0 &   0&    0 &   0 &   0 &  1 &  -2 & 1\\
      0 & 0  & 0 &  0 & 0 & 0 & 1& -2\end{array}\right)$$
is the (negative) Cartan matrix. The intersection form \eqref{ccet}
is even.

The {\it divisibility} $m(L)$ is
the largest positive integer which divides the lattice
element $L\in H^2(X,\mathbb{Z})$.
If the divisibility is 1,
$L$ is {\it primitive}.
Elements with
equal divisibility and norm square are equivalent up to orthogonal transformation 
of $H^2(X,\mathbb{Z})$, see \cite{CTC}.

\subsection{Lattice polarization} \label{lpol}
A primitive class 
$H\in \text{Pic}(X)$ is a {\it quasi-polarization}
if
$$\langle H,H \rangle_X >0  \ \ \ \text{and} \ \  \ \langle H,[C]\rangle_X
 \geq 0 $$
for every curve $C\subset X$.
A sufficiently high tensor power $H^n$
of a quasi-polarization is base point free and determines
a birational morphism
$$X\rightarrow \widetilde{X}$$
contracting A-D-E configurations of $(-2)$-curves on $X$.
Therefore, every quasi-polarized $K3$ surface is algebraic.

Let $\Lambda$ be a fixed rank $r$  
primitive{\footnote{A sublattice
is primitive if the quotient is torsion free.}}
sublattice
\begin{equation*} 
\Lambda \subset U\oplus U \oplus U \oplus E_8(-1) \oplus E_8(-1)
\end{equation*}
with signature $(1,r-1)$, and
let 
$v_1,\ldots, v_r \in \Lambda$ be an integral basis.
The discriminant is
$$\Delta(\Lambda) = (-1)^{r-1} \det
\begin{pmatrix}
\langle v_{1},v_{1}\rangle & \cdots & \langle v_{1},v_{r}\rangle  \\
\vdots & \ddots & \vdots \\
\langle v_{r},v_{1}\rangle & \cdots & \langle v_{r},v_{r}\rangle 
\end{pmatrix}\,.$$
The sign is chosen so $\Delta(\Lambda)>0$.

A {\it $\Lambda$-polarization} of a $K3$ surface $X$   
is a primitive embedding
$$j: \Lambda \hookrightarrow \mathrm{Pic}(X)$$ 
satisfying two properties:
\begin{enumerate}
\item[(i)] the lattice pairs 
$\Lambda \subset U^3\oplus E_8(-1)^2$ and
$\Lambda\subset 
H^2(X,\mathbb{Z})$ are isomorphic
via an isometry which restricts to the identity on $\Lambda$,
\item[(ii)]
$\text{Im}(j)$ contains
a {quasi-polarization}. 
\end{enumerate}
By (ii), every $\Lambda$-polarized $K3$ surface is algebraic.

The period domain $M$ of Hodge structures of type $(1,20,1)$ on the lattice 
$U^3 \oplus E_8(-1)^2$ is
an analytic open subset
of the 20-dimensional nonsingular isotropic 
quadric $Q$,
$$M\subset Q\subset \proj\big( (U^3 \oplus E_8(-1)^2 )    
\otimes_\Z \com\big)\,.$$
Let $M_\Lambda\subset M$ be the locus of vectors orthogonal to 
the entire sublattice $\Lambda \subset U^3 \oplus E_8(-1)^2$.

Let $\Gamma$ be the isometry group of the lattice 
$U^3 \oplus E_8(-1)^2$, and let
$$\Gamma_\Lambda \subset \Gamma$$ be the
subgroup  restricting to the identity on $\Lambda$.
By global Torelli,
the moduli space~$\mathcal{M}_{\Lambda}$ of $\Lambda$-polarized $K3$ surfaces 
is the quotient
$$\mathcal{M}_\Lambda = M_\Lambda/\Gamma_\Lambda\,.$$
We refer the reader to \cite{dolga} for a detailed
discussion.


\subsection{Genus $0$ invariants}\label{gen0}

Let 
$L\in \text{Pic}(X)$ be a nonzero and {\it admissible} class
on a $K3$ surface $X$ as defined in Section \ref{hah2}:
\begin{enumerate}
\item[(i)] $\frac{1}{m(L)^2}\cdot \langle L,L\rangle_X \geq -2$,
\item[(ii)] $\langle H, L\rangle_X \geq 0$.
\end{enumerate}
In case of equalities in both (i) and (ii), we further require $L$ to be effective.

\begin{proposition} \label{trtr} The reduced genus $0$ Gromov-Witten invariant
$$N_0(L) = \int_{[\MM_{0,0}(X,L)]^{\textup{red}}} 1$$
is nonzero for all admissible classes $L$.
\end{proposition}

\begin{proof} The result is a direct consequence of the full Yau-Zaslow formula (including
multiple classes) proven in \cite{KMPS}.
We define $N_0(\ell)$ for $\ell\geq-1$ by
$$\sum_{\ell=-1}^\infty q^\ell N_0(\ell)\, = \, \frac{1}{q\prod_{n=1}^\infty (1-q^n)^{24}}\, =\, 
\frac{1}{q}+ 24 + 324 q + 3200 q^2 \ldots\, .$$
For $\ell<-1$, we set $N_0(\ell)=0$. 
By the full Yau-Zaslow formula,
\begin{equation} \label{pqq2}
N_0(L)
= \sum_{r|m(L)} \frac{1}{r^3} 
N_0\left(\frac{ \langle L,L\rangle_X}{2r^2} \right)\, .
\end{equation}
Since all $N_0(\ell)$ for $\ell\geq -1$ are positive, the right side of \eqref{pqq2} is positive.
\end{proof}

\subsection{Genus $1$ invariants} \label{gen1}
Let 
$L\in \text{Pic}(X)$ be an {admissible} class
on a $K3$ surface $X$.
Let $$N_1(L) = \int_{[\MM_{1, 1}(X, L)]^{\text{red}}} \text{ev}^*(\mathsf{p})$$
be the reduced invariant virtually counting elliptic curves
passing through a point of $X$.
We define 
$$\sum_{\ell=0}^\infty q^\ell N_1(\ell) \, = \,  
\frac{
\sum_{k=1}^\infty \sum_{d|k} dk q^k
}{q\prod_{n=1}^\infty (1-q^n)^{24}}\,
\, =\, 1 + 30 q + 480 q^2 + 5460 q^3 \ldots\, .$$
For $\ell\leq -1$, we set $N_1(\ell)=0$.
If $L$ is primitive, 
$$N_1(L)=N_1\left(\frac{\langle L,L\rangle_X}{2}\right)$$
by a result of 
\cite{brl}. In particular,
$N_1(L)>0$ for $L$ admissible and primitive if $\langle L, L\rangle_X\geq 0$.

\begin{proposition}[Oberdieck] \label{trtrtr} The reduced genus $1$ Gromov-Witten invariant
$N_1(L)$
is nonzero for all admissible classes $L$ satisfying
$\langle L, L\rangle_X\geq 0$.
\end{proposition}

\begin{proof}
The result is a direct consequence of the multiple cover
formula for the reduced Gromov-Witten theory of
$K3$ surfaces conjectured in \cite{OP}.
By the multiple cover formula,
\begin{equation} \label{pqq3}
N_1(L) = 
\sum_{r|m(L)} r 
N_1\left(\frac{ \langle L,L\rangle_X}{2r^2} \right)\, .
\end{equation}
Since all $N_1(\ell)$ for $\ell\geq 0$ are positive, the right side of \eqref{pqq3} is positive.

To complete the argument, we must prove the multiple cover
formula \eqref{pqq3} in the required genus 1 case. 
We derive \eqref{pqq3} from  the genus 2 case of the
Katz-Klemm-Vafa formula for imprimitive classes proven in \cite{PT}.
Let
$$N_2(L)= \int_{[\MM_{2}(X, L)]^{\text{red}}} \lambda_2\, ,$$
where $\lambda_2$ is the pull-back of the second
Chern class of the Hodge bundle on $\MM_{2}$.
Using the well-known boundary expression{\footnote{See \cite{Mumford}.
A more recent approach valid also for higher genus can be found in
\cite{DRcycles}.}} for $\lambda_2$
in the tautological ring of $\MM_{2}$, Pixton \cite[Appendix]{MPT}
proves 
\begin{equation}\label{gww}
N_2(L)= \frac{1}{10} N_1(L) + \frac{\langle L,L\rangle^2_X}{960} N_0(L)\,.
\end{equation}
By \cite{PT}, the multiple cover formula for $N_2(L)$
carries a factor of $r$.  By the Yau-Zaslow formula
for imprimitive classes \cite{KMPS}, the term $\frac{\langle L,L\rangle^2_X}{960} N_0(L)$ also carries a factor of 
$$(r^2)^2\cdot \frac{1}{r^3} = r\,. $$
By \eqref{gww}, $N_1(L)$ must then  carry a factor
of $r$ in the multiple cover formula exactly
as claimed in \eqref{pqq3}.
\end{proof}

\subsection{Vanishing}

Let 
$L\in \text{Pic}(X)$ be an {inadmissible} class
on a $K3$ surface $X$. The following vanishing result holds.

\begin{proposition} \label{vvvv} For inadmissible $L$, the reduced virtual 
class 
is $0$ in Chow,
$$\left[\MM_{g,n}(X,L)\right]^{\textup{red}} \,=\, 0 \ \in\  \mathsf{A}_{g+n}(\MM_{g,n}(X,L),\mathbb{Q})\,.$$
\end{proposition}

\begin{proof}
Consider a 1-parameter family of $K3$ surfaces
\begin{equation}\label{q22q}
\pi_C:\mathcal{X} \rightarrow (C,0)
\end{equation}
with special fiber $\pi^{-1}(0)=X$ for which the class
$L$ is algebraic on all fibers. Let
\begin{equation}\label{trrtt} \phi:
\MM_{g,n}(\pi_C, L) \rightarrow C
\end{equation}
be the universal moduli space of stable maps to the fibers of $\pi_C$.
Let
$$\iota: 0 \hookrightarrow C$$
be the inclusion of the special point.
By the construction of the reduced class,
$$[\MM_{g,n}(X, L)]^{\text{red}} = \iota^! 
[\MM_{g,n}(\pi_C, L)]^{\text{red}}\, .$$

Using the argument of \cite[Lemma 2]{MP} for
elliptically fibered $K3$ surfaces with a section, such a family 
\eqref{q22q} can be found for which
the fiber of $\phi$ is {\it empty} over a general point
of~$C$ since $L$ is not generically
effective. The vanishing
\begin{equation}\label{dkkd}
\left[\MM_{g,n}(X,L)\right]^{\textup{red}} \,=\, 0 \ \in\  \mathsf{A}_{g+n}(\MM_{g,n}(X,L),\mathbb{Q})\,
\end{equation}
then follows: $\iota^!$ of {\it any} cycle which does not dominate $C$ is 0.

If the family \eqref{q22q} consists of projective $K3$ surfaces, the argument
stays within the Gromov-Witten theory of algebraic varieties.
However, if the family consists of non-algebraic $K3$ surfaces (as may be the case since
$L$ is not ample), a few more steps are needed.
First, we can assume {\it all} stable maps to the fiber of the family $\eqref{q22q}$
lie over $0\in C$ and map to the algebraic fiber $X$.
There is no difficulty in constructing the moduli space of stable
maps \eqref{trrtt}. In fact, all the geometry takes place over
an Artinian neighborhood of $0\in C$. Therefore the cones and intersection
theory are all algebraic. We conclude the vanishing \eqref{dkkd}.
\end{proof}

\section{Gromov-Witten theory for families of $K3$ surfaces} \label{zzss}
\subsection{The divisor $\mathcal{L}$} \label{fafa2}
Let $\mathcal{B}$ be any nonsingular base scheme, and let 
$$\pi_{\mathcal{B}}: \mathcal{X}_{\mathcal{B}} \rightarrow \mathcal{B}$$
be a family of $\Lambda$-polarized $K3$ surfaces.{\footnote{Since the 
quasi-polarization class may not be ample, $\mathcal{X}_{\mathcal{B}}$ may be a
nonsingular algebraic space. There is no difficulty in defining the
moduli space of stable maps and the associated virtual classes
for such nonsingular algebraic spaces. Since the stable maps are to the
fiber classes, the moduli spaces are of finite type.
In the original paper on virtual fundamental classes by Behrend and Fantechi \cite{BF},
the obstruction theory on the moduli space of stable maps was required
to have a global resolution (usually obtained from an ample 
bundle on the target). However, the global resolution hypothesis
was removed by Kresch in \cite[Theorem 5.2.1]{kresch}.
}}
For $L\in \Lambda$ admissible, consider the moduli space
\begin{equation}\label{eee}
\MM_{g,n}(\pi_{\mathcal{B}},L) \rightarrow \mathcal{B}\, .
\end{equation}
The relationship between the $\pi_{\mathcal{B}}$-relative standard and reduced
obstruction theory of 
$\MM_{g,n}(\pi_{\mathcal{B}},L)$ yields 
$$\left[\MM_{g,n}(\pi_{\mathcal{B}},L)\right]^{\text{vir}}= -\lambda \cdot \left[\MM_{g,n}(\pi_{\mathcal{B}},L)\right]^{\text{red}}\, $$
where $\lambda$ is the pull-back via \eqref{eee} of the Hodge bundle on $\mathcal{B}$.
The reduced class is of $\pi_{\mathcal{B}}$-relative dimension $g+n$.

The canonical divisor class associated to an admissible $L\in \Lambda$ is 
$$\mathcal{L} \, = \, \frac{1}{N_0(L)} \,\cdot \, 
\epsilon_*\left[ \MM_{0,1}(\pi_{\mathcal{B}},L)\right]^{\text{red}} \ \in \ \mathsf{A}^1(\mathcal{X}_{\mathcal{B}},\mathbb{Q}) \, .$$
By Proposition \ref{trtr},
the reduced Gromov-Witten invariant
$$N_0(L) = \int_{[\MM_{0,0}(X,L)]^{\text{red}}} 1$$
is not zero.

For a family of $\Lambda$-polarized $K3$ surfaces
over any base scheme $\mathcal{B}$, we define
$$\mathcal{L} \ \in \ \mathsf{A}^1(\mathcal{X}_{\mathcal{B}},\mathbb{Q})$$
by pull-back from the universal family over the nonsingular
moduli stack $\mathcal{M}_\Lambda$.

\subsection{The divisor $\w{\mathcal{L}}$} \label{gwgw}
Let $\mathcal{X}_\Lambda$ denote the
universal $\Lambda$-polarized $K3$ surface over $\mathcal{M}_{\Lambda}$,
$$\pi_\Lambda: \mathcal{X}_\Lambda \rightarrow \mathcal{M}_{\Lambda}\, .$$
For $L\in \Lambda$  admissible, 
Let
$\MM_{0,0}(\pi_\Lambda, L)$
be the $\pi_\Lambda$-relative moduli space of genus 0 stable maps.
Let $$\phi: \MM_{0,0}(\pi_\Lambda, L) \rightarrow \mathcal{M}_\Lambda$$
be the proper structure map.
The reduced virtual class
$\left[\MM_{0,0}(\pi_\Lambda, L)\right]^{\text{red}}$
is of $\phi$-relative dimension 0 and
satisfies
$$\phi_*\left[\MM_{0,0}(\pi_\Lambda, L)\right]^{\text{red}} \, =\,  N_0(L) \cdot [\mathcal{M}_\Lambda]\, \neq\, 0 \,.$$

The universal curve over the moduli space of stable maps,
$$\mathsf{C} \rightarrow
\MM_{0,0}(\pi_\Lambda, L) \,,$$
carries an evaluation morphism
$$\epsilon_{\MM}: \mathsf{C} \rightarrow \mathcal{X}_\MM \,=\, \phi^*\mathcal{X}_\Lambda$$
over $\mathcal{M}_{\Lambda}$. Via the Hilbert-Chow map, the image of $\epsilon_{\MM}$ determines
a canonical Chow cohomology class
$$\widehat{\mathcal{L}} \ \in \ \mathsf{A}^1( \mathcal{X}_\MM,\mathbb{Q})\,. $$
Via pull-back, we also have the class
$${\mathcal{L}} \ \in \ \mathsf{A}^1( \mathcal{X}_\MM,\mathbb{Q})\, $$
constructed in Section \ref{fafa2}. 

The classes $\widehat{\mathcal{L}}$
and
$\mathcal{L}$ are
are certainly equal when restricted to the fibers of 
$$\pi_{\MM} \, : \mathcal{X}_{\MM} \, \rightarrow\, 
\MM_{0,0}(\pi_\Lambda, L)\, .$$
However, more is true. We define the reduced virtual class
of $\mathcal{X}_\MM$ by flat pull-back,
$$[\mathcal{X}_\MM]^{\text{red}} \, = \, \pi_\MM^*\, \left[\MM_{0,0}(\pi_\Lambda, L)\right]^{\text{red}}\ \in \ 
\mathsf{A}_{\mathsf{d}(\Lambda)+2}( \mathcal{X}_\MM,\mathbb{Q})\, ,$$
where $\mathsf{d}(\Lambda)=20- \text{rank}(\Lambda)$ is the dimension
of $\mathcal{M}_\Lambda$.

\begin{theorem} \label{dldl} 
For $L\in \Lambda$ admissible,
$$\w{\mathcal{L}}\cap [\mathcal{X}_\MM]^{\textup{red}}
\, =\, \mathcal{L} \cap [\mathcal{X}_\MM]^{\textup{red}}\ \in\  \mathsf{A}_{\mathsf{d}(\Lambda)+1}( \mathcal{X}_\MM,\mathbb{Q})\, .$$
\end{theorem}

\noindent The proof of Theorem \ref{dldl} will be given in Section \ref{pfpf5}.

\section{Basic push-forwards in genus $0$ and $1$} \label{expo}

\subsection{Push-forwards of reduced classes}
Let $L \in \Lambda$ be a nonzero class. As discussed in Section \ref{conjs}, the export construction requires knowing the push-forward of the reduced virtual class $\left[\MM_{g,n}(\pi_\Lambda,L)\right]^{\text{red}}$ via the evaluation map
$$\epsilon^n: \MM_{g,n}(\pi_\Lambda,L) \rightarrow \mathcal{X}^n_\Lambda\, .$$
Fortunately, to export the WDVV and Getzler relations, we only need to analyze
three simple cases.

\subsection{Case $g = 0$, $n \geq 1$}

Consider the push-forward class in genus 0,
$$\epsilon^n_*\left[\MM_{0,n}(\pi_\Lambda,L)\right]^{\text{red}} \ \in \ \mathsf{A}^n(\mathcal{X}_\Lambda^n, \mathbb{Q}) \,.$$
For $n = 1$ and $L \in \Lambda$ admissible, we have by definition
$$\epsilon_*\left[\MM_{0,1}(\pi_\Lambda,L)\right]^{\text{red}} \, = \, N_0(L) \cdot \mathcal{L} \,.$$

\begin{proposition}\label{zzzr}
For all $n \geq 1$, we have
$$\epsilon^n_*\left[\MM_{0,n}(\pi_\Lambda,L)\right]^{\textup{red}} \, =
\begin{cases}
\, N_0(L) \cdot \mathcal{L}_{(1)} \cdots \mathcal{L}_{(n)} & \text{ if $L \in \Lambda$ is admissible} \,, \\
\, 0 & \text{ if not} \,.
\end{cases}$$
Here $\mathcal{L}_{(i)}$ is the pull-back of $\mathcal{L}$ via the $i^{\text{th}}$ projection.
\end{proposition}

\begin{proof}
Consider first the case where the class $L\in \Lambda$ is admissible.
The evaluation map~$\epsilon^n$ factors as
$$\MM_{0,n}(\pi_\Lambda,L) \, \stackrel{\epsilon^n_{\MM}}{\longrightarrow} \, 
\mathcal{X}_\MM^n \, \stackrel{\rho^n}{\longrightarrow}\, \mathcal{X}_\Lambda^n$$
where $\epsilon^n_\MM$ is the lifted evaluation map and $\rho^n$
is the projection. We have
\begin{align*}
\epsilon^n_*\left[\MM_{0,n}(\pi_\Lambda,L)\right]^{\text{red}} \, & = \,
\rho^n_*\epsilon_{\MM*}^n \left[\MM_{0,n}(\pi_\Lambda,L)\right]^{\text{red}} \\
& = \, \rho^n_* \left(\w{\mathcal{L}}_{(1)} \cdots \w{\mathcal{L}}_{(n)} \cap [\mathcal{X}_\MM^n]^{\text{red}}\right) \\
& = \, \rho^n_* \left({\mathcal{L}}_{(1)} \cdots {\mathcal{L}}_{(n)} \cap
[\mathcal{X}_\MM^n]^{\text{red}}\right) \\
& = \, N_0(L)\cdot {\mathcal{L}}_{(1)} \cdots {\mathcal{L}}_{(n)} \cap
[\mathcal{X}^n_\Lambda]\, , 
\end{align*}
where the third equality is a consequence of Theorem \ref{dldl}.

Next, consider the case where $L \in \Lambda$ is inadmissible. By Proposition \ref{vvvv} and a spreading out argument \cite[1.1.2]{Voi}, the reduced class $\left[\MM_{0, n}(\pi_\Lambda, L)\right]^{\text{red}}$ is supported over a proper subset
of $\mathcal{M}_\Lambda$. Since $K3$ surfaces are not ruled, the support of
$$\epsilon^n_*\left[\MM_{0,n}(\pi_\Lambda,L)\right]^{\text{red}} \ \in \ \mathsf{A}^n(\mathcal{X}_\Lambda^n, \mathbb{Q})$$
has codimension at least $n + 1$  and therefore  vanishes.
\end{proof}

\subsection{Case $g = 1$, $n= 1$}

The push-forward class
$$\epsilon_*\left[\MM_{1,1}(\pi_\Lambda,L)\right]^{\text{red}} \ \in \ \mathsf{A}^0(\mathcal{X}_\Lambda, \mathbb{Q})$$
is a multiple of the fundamental class of $\mathcal{X}_\Lambda$. 
\begin{proposition} \label{g1p1}
We have
\begin{equation*}
\epsilon_*\left[\MM_{1,1}(\pi_\Lambda,L)\right]^{\textup{red}}
= \begin{cases}
\, N_1(L) \cdot [\mathcal{X}_\Lambda]  & \text{ if $L \in \Lambda$ is admissible and $\langle L, L\rangle_\Lambda \geq 0$} \,, \\
\, 0 & \text{ if not} \,.
\end{cases}
\end{equation*}
\end{proposition}

\begin{proof}
The multiple of the fundamental class $[\mathcal{X}_\Lambda]$ can be computed fiberwise:
it is the genus 1 Gromov-Witten invariant
$$N_1(L) = \int_{[\MM_{1, 1}(X, L)]^{\text{red}}} \text{ev}^*(\mathsf{p}) \, .$$
The invariant vanishes for $L \in \text{Pic}(X)$ inadmissible as well as for $L$ admissible and $\langle L, L\rangle_X < 0$.
\end{proof}

\subsection{Case $g = 1$, $n= 2$}
The push-forward class is a divisor,
$$\epsilon^2_*\left[\MM_{1,2}(\pi_\Lambda,L)\right]^{\text{red}} \ \in \ \mathsf{A}^1(\mathcal{X}^2_\Lambda, \mathbb{Q}) \,.$$

\begin{proposition} \label{g1g1}
We have
\begin{multline*}
\epsilon^2_*\left[\MM_{1,2}(\pi_\Lambda,L)\right]^{\textup{red}} \\ 
= \begin{cases}
\, N_1(L) \cdot \Big(\mathcal{L}_{(1)} + \mathcal{L}_{(2)} + Z(L)\Big) & \text{ if $L \in \Lambda$ is admissible and $\langle L, L\rangle_\Lambda \geq 0$} \,, \\
\, 0 & \text{ if not} \,.
\end{cases}
\end{multline*}
Here $Z(L)$ is a divisor class in $\mathsf{A}^1(\mathcal{M}_\Lambda, \mathbb{Q})$ depending on $L$.\footnote{We identify $\AA(\mathcal{M}_\Lambda, \mathbb{Q})$ as a subring of $\AA(\mathcal{X}^n_\Lambda, \mathbb{Q})$ via $\pi_\Lambda^{n*}$.}
\end{proposition}

\noindent In Section \ref{dvdv}, we will compute $Z(L)$ explicitly in terms of Noether-Lefschetz divisors in the moduli
space $\mathcal{M}_\Lambda$.

\begin{proof}
Consider first the case where the class
$L\in \Lambda$ is admissible and $\langle L, L\rangle_\Lambda \geq 0$. If $L$ is a multiple of the quasi-polarization $H$, we may assume $\Lambda = (2\ell)$. Then, the relative Picard group
$$\text{Pic}(\mathcal{X}_\Lambda/\mathcal{M}_\Lambda)$$
has rank 1. Since the reduced class $\left[\MM_{1, 2}(\pi_\Lambda, L)\right]^{\text{red}}$ is $\mathfrak{S}_2$-invariant, 
the push-forward takes the form
\begin{equation} \label{fofofo}
\epsilon^2_*\left[\MM_{1,2}(\pi_\Lambda,L)\right]^{\textup{red}} \, = \, c(L) \cdot \left(\mathcal{L}_{(1)} + \mathcal{L}_{(2)}\right) + \widetilde{Z}(L) \ \in \ \mathsf{A}^1(\mathcal{X}^2_\Lambda, \mathbb{Q}) \, ,
\end{equation}
where $c(L) \in \mathbb{Q}$ and $\widetilde{Z}(L)$ is (the pull-back of) a divisor class
 in $\mathsf{A}^1(\mathcal{M}_\Lambda, \mathbb{Q})$. 

The constant $c(L)$ can be computed fiberwise: by the divisor equation{\footnote{Since
$L$ is a multiple of the quasi-polarization, $\langle L, L\rangle_\Lambda > 0$.}}, we have
$$c(L) = N_1(L) \,.$$
Since $N_1(L) \neq 0$ by Proposition \ref{trtrtr}, we can rewrite \eqref{fofofo} as
$$\epsilon^2_*\left[\MM_{1,2}(\pi_\Lambda,L)\right]^{\textup{red}} \, = \, N_1(L) \cdot \Big(\mathcal{L}_{(1)} + \mathcal{L}_{(2)} + Z(L)\Big) \ \in \ \mathsf{A}^1(\mathcal{X}^2_\Lambda, \mathbb{Q}) \,,$$
where $Z(L) \in \mathsf{A}^1(\mathcal{M}_\Lambda, \mathbb{Q})$.

If $L \neq m \cdot H$, we may assume $\Lambda$ to be a rank 2 lattice with $H,L\in \Lambda$.
Then, the push-forward class takes the form
\begin{multline} \label{rmrmrm}
\epsilon^2_*\left[\MM_{1,2}(\pi_\Lambda,L)\right]^{\textup{red}} \, = \, c_H(L) \cdot \left(\mathcal{H}_{(1)} + \mathcal{H}_{(2)}\right) + c_L(L) \cdot \left(\mathcal{L}_{(1)} + \mathcal{L}_{(2)}\right) \\
+ \widetilde{Z}(L) \ \in \ \mathsf{A}^1(\mathcal{X}^2_\Lambda, \mathbb{Q}) \,,
\end{multline}
where $c_H(L), c_L(L) \in \mathbb{Q}$ and $\widetilde{Z}(L) \in \mathsf{A}^1(\mathcal{M}_\Lambda, \mathbb{Q})$. By applying the divisor equation with respect to
$$\langle L, L\rangle_\Lambda \cdot H - \langle H, L\rangle_\Lambda \cdot L \,,$$
we find
$$c_H(L) \Big(2\ell \langle L, L\rangle_\Lambda - \langle H, L\rangle_\Lambda^2\Big) = 0 \,.$$
Since $2\ell \langle L, L\rangle_\Lambda - \langle H, L\rangle_\Lambda^2 < 0$
by the Hodge index theorem, we have $c_H(L) = 0$. Moreover, by applying the divisor equation with respect to $H$, we find
$$c_L(L) = N_1(L) \,.$$
Since $N_1(L) \neq 0$ by Proposition \ref{trtrtr}, we can rewrite \eqref{rmrmrm} as
$$\epsilon^2_*\left[\MM_{1,2}(\pi_\Lambda,L)\right]^{\textup{red}} \, = \, N_1(L) \cdot \Big(\mathcal{L}_{(1)} + \mathcal{L}_{(2)} + Z(L)\Big) \ \in \ \mathsf{A}^1(\mathcal{X}^2_\Lambda, \mathbb{Q}) \,,$$
where $Z(L) \in \mathsf{A}^1(\mathcal{M}_\Lambda, \mathbb{Q})$.

Next, consider the case where the class 
$L \in \Lambda$ is inadmissible. As before, by Proposition \ref{vvvv} and a spreading out argument, the reduced class $\left[\MM_{1, 2}(\pi_\Lambda, L)\right]^{\text{red}}$ is supported over a proper subset of $\mathcal{M}_\Lambda$. Since $K3$ surfaces are not elliptically connected\footnote{A nonsingular projective variety $Y$ is said to be {\it elliptically connected} if there is a genus 1 curve passing through two general points of $Y$. In dimension $\geq 2$, elliptically connected varieties are uniruled, see \cite[Proposition 6.1]{Gou}.}, the support of the push-forward class
$$\epsilon^2_*\left[\MM_{1,2}(\pi_\Lambda, L)\right]^{\text{red}} \ \in \ \mathsf{A}^1(\mathcal{X}_\Lambda^2, \mathbb{Q})$$
has codimension at least 2. Hence, the push-forward class vanishes.

Finally, for $L \in \Lambda$ admissible and $\langle L, L\rangle_\Lambda < 0$, the reduced class $\left[\MM_{1,2}(\pi_\Lambda, L)\right]^{\text{red}}$ is fiberwise supported on the products of finitely many curves in the $K3$ surface.{\footnote{The proof exactly follows
the argument of Proposition \ref{vvvv}. We find a (possibly non-algebraic) 1-parameter family of $K3$
surfaces for which the class $L$ is generically a multiple of a $(-2)$-curve. The open moduli space
of stable maps to the $K3$ fibers 
which are not supported on the family
of $(-2)$-curves (and its  
limit curve in the special fiber) is constrained
to lie over the special point in the base
of the family.
The specialization argument
of Proposition \ref{vvvv} then shows
the virtual class is 0 when restricted to 
the open moduli space of stable maps
to the special fiber which are not supported
on the limit curve.}} This implies the support of the push-forward class $\epsilon^2_*\left[\MM_{1,2}(\pi_\Lambda, L)\right]^{\text{red}}$ has codimension 2 in~$\mathcal{X}_\Lambda^2$. Hence, the push-forward class vanishes.
\end{proof}

\section{Exportation of the WDVV relation} \label{wwww}
\subsection{Exportation} 
Let $L \in \Lambda$ be an admissible class. 
Consider the
morphisms
$$ \MM_{0,4}\ \stackrel{\tau}{\longleftarrow} \ \MM_{0,4}(\pi_\Lambda, L) \ 
\stackrel{\epsilon^4}{\longrightarrow}\  
\mathcal{X}^4_{\Lambda}\, .$$
Following the notation of Section \ref{conjs}, we export here the WDVV relation with respect to 
the curve class $L$,
\begin{equation}\label{exex}
\epsilon_*^4 \tau^*(\mathsf{WDVV}) \,=\, 0 \ \in\ \mathsf{A}^5
(\mathcal{X}^4_\Lambda, \mathbb{Q})\, .
\end{equation}
We will compute $\epsilon_*^4 \tau^*(\mathsf{WDVV})$ by applying the splitting
axiom of Gromov-Witten theory to the two terms of the WDVV relation \eqref{wdvv}.
The splitting axiom requires a distribution of the curve class
to each vertex of each graph appearing in \eqref{wdvv}.

\subsection{WDVV relation: unsplit contributions}

The unsplit contributions are obtained  from curve class distributions 
which do {\it not} split $L$. The first unsplit
contributions come from the first graph of \eqref{wdvv}:
$$\left[\begin{tikzpicture}[baseline={([yshift=-.5ex]current bounding box.center)}]
	\node[leg] (l3) at (0,2) [label=above:$3$] {};
	\node[leg] (l4) at (1,2) [label=above:$4$] {};
	\node[vertex] (v2) at (.5,1.5) [label=right:$0$] {};
	\node[circ] (v1) at (.5,.5) [label=right:$0$] {$L$};
	\node[leg] (l1) at (0,0) [label=below:$1$] {};
	\node[leg] (l2) at (1,0) [label=below:$2$] {};
	\path
		(l3) edge (v2)
		(l4) edge (v2)
		(v2) edge (v1)
		(v1) edge (l1)
		(v1) edge (l2)
	;
\end{tikzpicture}\right] \ + \ 
\left[\begin{tikzpicture}[baseline={([yshift=-.5ex]current bounding box.center)}]
	\node[leg] (l3) at (0,2) [label=above:$3$] {};
	\node[leg] (l4) at (1,2) [label=above:$4$] {};
	\node[circ] (v2) at (.5,1.5) [label=right:$0$] {$L$};
	\node[vertex] (v1) at (.5,.5) [label=right:$0$] {};
	\node[leg] (l1) at (0,0) [label=below:$1$] {};
	\node[leg] (l2) at (1,0) [label=below:$2$] {};
	\path
		(l3) edge (v2)
		(l4) edge (v2)
		(v2) edge (v1)
		(v1) edge (l1)
		(v1) edge (l2)
	;
\end{tikzpicture}\right]$$
$$N_0(L) \cdot \Big({\mathcal{L}}_{(1)} {\mathcal{L}}_{(2)} {\mathcal{L}}_{(3)} \Delta_{(34)} + {\mathcal{L}}_{(1)}{\mathcal{L}}_{(3)} {\mathcal{L}}_{(4)} \Delta_{(12)}\Big) \,.$$

\vspace{8pt}
\noindent The unsplit contributions  from the second graph of \eqref{wdvv} are:
$$- \left[\begin{tikzpicture}[baseline={([yshift=-.5ex]current bounding box.center)}]
	\node[leg] (l2) at (0,2) [label=above:$2$] {};
	\node[leg] (l4) at (1,2) [label=above:$4$] {};
	\node[vertex] (v2) at (.5,1.5) [label=right:$0$] {};
	\node[circ] (v1) at (.5,.5) [label=right:$0$] {$L$};
	\node[leg] (l1) at (0,0) [label=below:$1$] {};
	\node[leg] (l3) at (1,0) [label=below:$3$] {};
	\path
		(l2) edge (v2)
		(l4) edge (v2)
		(v2) edge (v1)
		(v1) edge (l1)
		(v1) edge (l3)
	;
\end{tikzpicture}\right] \ - \ 
\left[\begin{tikzpicture}[baseline={([yshift=-.5ex]current bounding box.center)}]
	\node[leg] (l2) at (0,2) [label=above:$2$] {};
	\node[leg] (l4) at (1,2) [label=above:$4$] {};
	\node[circ] (v2) at (.5,1.5) [label=right:$0$] {$L$};
	\node[vertex] (v1) at (.5,.5) [label=right:$0$] {};
	\node[leg] (l1) at (0,0) [label=below:$1$] {};
	\node[leg] (l3) at (1,0) [label=below:$3$] {};
	\path
		(l2) edge (v2)
		(l4) edge (v2)
		(v2) edge (v1)
		(v1) edge (l1)
		(v1) edge (l3)
	;
\end{tikzpicture}\right]$$
$$-N_0(L) \cdot \Big({\mathcal{L}}_{(1)}{\mathcal{L}}_{(2)}{\mathcal{L}}_{(3)}\Delta_{(24)} + {\mathcal{L}}_{(1)} {\mathcal{L}}_{(2)}{\mathcal{L}}_{(4)}\Delta_{(13)}\Big) \,.$$

\vspace{8pt}
The curve class $0$ vertex is not reduced and yields the usual intersection form (which explains
the presence of diagonal $\Delta_{(ij)}$). The curve class $L$ vertex is reduced. We have applied
Proposition \ref{zzzr} to compute the push-forward to $\mathcal{X}^4_\Lambda$.
All terms are of relative codimension 5 (codimension 1 each for the factors $\mathcal{L}_{(i)}$
and codimension 2 for the diagonal~$\Delta_{(ij)}$).
The four unsplit terms (divided by $N_0(L)$) exactly 
constitute the principal part of Theorem \ref{WDVV}.

\subsection{WDVV relation: split contributions}

The split contributions are obtained
from non-trivial curve class distributions to the vertices
$$L = L_1 + L_2\, , \ \  L_1\,,\,L_2\neq 0\, .$$ 
By Proposition \ref{zzzr}, we need only consider distributions where {\it both} $L_1$
and $L_2$ are admissible classes.
Let $\ww\Lambda$ be the saturation{\footnote{We work only with primitive sublattices of $U^3 \oplus E_8(-1)^2$.}} of the span of $L_1$, $L_2$, and $\Lambda$.
There are two types.

\vspace{8pt}
\noindent \makebox[12pt][l]{$\bullet$}If $\text{rank}(\ww{\Lambda}) = \text{rank}(\Lambda) + 1$, the 
split contributions are pushed forward from $\mathcal{X}^4_{\ww{\Lambda}}$ via the map 
$\mathcal{X}^4_{\ww{\Lambda}} \to \mathcal{X}^4_\Lambda$.
Both vertices carry the reduced class by the obstruction calculation 
of \cite[Lemma 1]{MP}.
The split contributions are:
$$\left[
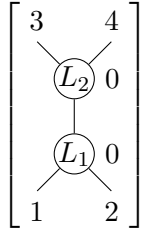
\begin{tikzpicture}[baseline={([yshift=-.5ex]current bounding box.center)}]
	\node[leg] (l3) at (0,2) [label=above:$3$] {};
	\node[leg] (l4) at (1,2) [label=above:$4$] {};
	\node[circ] (v2) at (.5,1.5) [label=right:$0$] {$L_2$};
	\node[circ] (v1) at (.5,.5) [label=right:$0$] {$L_1$};
	\node[leg] (l1) at (0,0) [label=below:$1$] {};
	\node[leg] (l2) at (1,0) [label=below:$2$] {};
	\path
		(l3) edge (v2)
		(l4) edge (v2)
		(v2) edge (v1)
		(v1) edge (l1)
		(v1) edge (l2)
	;
\end{tikzpicture}\right]$$
\vspace{0pt}
$$N_0(L_1) N_0(L_2) \langle L_1, L_2\rangle_{\ww{\Lambda}} \cdot \mathcal{L}_{1,(1)}\mathcal{L}_{1,(2)}\mathcal{L}_{2,(3)}\mathcal{L}_{2,(4)}\, ,$$

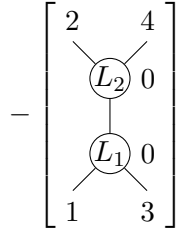
$$- \left[\begin{tikzpicture}[baseline={([yshift=-.5ex]current bounding box.center)}]
	\node[leg] (l2) at (0,2) [label=above:$2$] {};
	\node[leg] (l4) at (1,2) [label=above:$4$] {};
	\node[circ] (v2) at (.5,1.5) [label=right:$0$] {$L_2$};
	\node[circ] (v1) at (.5,.5) [label=right:$0$] {$L_1$};
	\node[leg] (l1) at (0,0) [label=below:$1$] {};
	\node[leg] (l3) at (1,0) [label=below:$3$] {};
	\path
		(l2) edge (v2)
		(l4) edge (v2)
		(v2) edge (v1)
		(v1) edge (l1)
		(v1) edge (l3)
	;
\end{tikzpicture}\right]$$
\vspace{0pt}
$$-N_0(L_1) N_0(L_2) \langle L_1, L_2\rangle_{\ww{\Lambda}} \cdot \mathcal{L}_{1,(1)}\mathcal{L}_{1,(3)}\mathcal{L}_{2,(2)}\mathcal{L}_{2,(4)}\, .$$

\vspace{8pt}
\noindent All terms are of relative codimension 5 (codimension 1 for the Noether-Lefschetz
condition and codimension 1 each for the factors $\mathcal{L}_{a,(i)}$).

\vspace{8pt}
\noindent \makebox[12pt][l]{$\bullet$}If $\ww{\Lambda} = \Lambda$, there is no obstruction
cancellation as above. The extra reduction yields a factor of $-\lambda$.
The split contributions are:
$$\left[
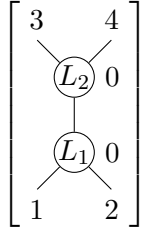
\begin{tikzpicture}[baseline={([yshift=-.5ex]current bounding box.center)}]
	\node[leg] (l3) at (0,2) [label=above:$3$] {};
	\node[leg] (l4) at (1,2) [label=above:$4$] {};
	\node[circ] (v2) at (.5,1.5) [label=right:$0$] {$L_2$};
	\node[circ] (v1) at (.5,.5) [label=right:$0$] {$L_1$};
	\node[leg] (l1) at (0,0) [label=below:$1$] {};
	\node[leg] (l2) at (1,0) [label=below:$2$] {};
	\path
		(l3) edge (v2)
		(l4) edge (v2)
		(v2) edge (v1)
		(v1) edge (l1)
		(v1) edge (l2)
	;
\end{tikzpicture}\right]$$
\vspace{0pt}
$$N_0(L_1) N_0(L_2) \langle L_1, L_2\rangle_{\ww{\Lambda}} \cdot (-\lambda)\mathcal{L}_{1,(1)}\mathcal{L}_{1,(2)}\mathcal{L}_{2,(3)}\mathcal{L}_{2,(4)}\, ,$$

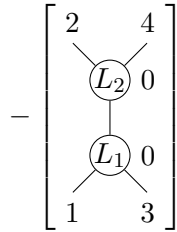
$$- \left[\begin{tikzpicture}[baseline={([yshift=-.5ex]current bounding box.center)}]
	\node[leg] (l2) at (0,2) [label=above:$2$] {};
	\node[leg] (l4) at (1,2) [label=above:$4$] {};
	\node[circ] (v2) at (.5,1.5) [label=right:$0$] {$L_2$};
	\node[circ] (v1) at (.5,.5) [label=right:$0$] {$L_1$};
	\node[leg] (l1) at (0,0) [label=below:$1$] {};
	\node[leg] (l3) at (1,0) [label=below:$3$] {};
	\path
		(l2) edge (v2)
		(l4) edge (v2)
		(v2) edge (v1)
		(v1) edge (l1)
		(v1) edge (l3)
	;
\end{tikzpicture}\right]$$
\vspace{0pt}
$$-N_0(L_1) N_0(L_2) \langle L_1, L_2\rangle_{\ww{\Lambda}} \cdot (-\lambda)\mathcal{L}_{1,(1)}\mathcal{L}_{1,(3)}\mathcal{L}_{2,(2)}\mathcal{L}_{2,(4)}\, .$$

\vspace{8pt}
\noindent All terms are of relative codimension 5 (codimension 1 for $-\lambda$
 and codimension 1 each for the factors $\mathcal{L}_{a,(i)}$).

\subsection{Proof of Theorem \ref{wdvv}}
The complete exported relation \eqref{exex}
is obtained by adding the 
unsplit contributions to the summation over all
split contributions $$L=L_1+L_2$$ of both types.
Split contributions of the first type are explicitly supported over the
Noether-Lefschetz locus corresponding to 
$$\ww{\Lambda} \subset U^3 \oplus E_8^2\, .$$
Split contributions of the second type all contain the factor $-\lambda$.
The class $\lambda$ is known to be a linear combination of proper
Noether-Lefschetz divisors of $\mathcal{M}_\Lambda$ by \cite[Theorem~1.2]{BKPSB}.
Hence, we view the split contributions of the second type also as
being supported over Noether-Lefschetz loci. 
For the formula
of Theorem \ref{WDVV}, we normalize the relation
by dividing by $N_0(L)$.
\qed

\section{Proof of Theorem \ref{dldl}} \label{pfpf5}


\subsection{Overview}
Let $L \in \Lambda$ be an admissible class, and let $\MM_{0,0}(\pi_\Lambda, L)$ be the $\pi_\Lambda$-relative moduli space of genus 0 stable maps,
$$\phi : \MM_{0,0}(\pi_\Lambda, L) \to \mathcal{M}_\Lambda\,.$$
Let $\mathcal{X}_{\MM}$ be the universal $\Lambda$-polarized $K3$ surface over $\MM_{0,0}(\pi_\Lambda, L)$,
$$\pi_{\MM} : \mathcal{X}_{\MM} \to \MM_{0,0}(\pi_\Lambda, L) \,.$$
In Sections \ref{fafa2} and \ref{gwgw}, we have constructed two divisor classes
$$\w{\mathcal{L}} \,, \ \mathcal{L} \ \in \ \mathsf{A}^1(\mathcal{X}_{\MM}, \mathbb{Q}) \,.$$
We define the $\kappa$ classes with respect to $\w{\mathcal{L}}$ by
$$\w{\kappa}_{[L^a;b]} \,=\, \pi_{\MM*}\left(\w{\mathcal{L}}^a \cdot c_2(\mathcal{T}_{\pi_{\MM}})^b\right) \ \in \ \mathsf{A}^{a + 2b - 2}\left(\MM_{0,0}(\pi_\Lambda, L), \mathbb{Q}\right) \,.$$

Since $\w{\mathcal{L}}$ and $\mathcal{L}$
are equal on the fibers of $\pi_{\MM}$,
the difference $\w{\mathcal{L}} - \mathcal{L}$ is the pull-back{\footnote{We use here
the vanishing $H^1(X,\mathcal{O}_X)=0$
for $K3$ surfaces $X$ and the base change
theorem.}} of a divisor class in $\mathsf{A}^1\left(\MM_{0,0}(\pi_\Lambda, L), \mathbb{Q}\right)$. In fact,
the difference is equal{\footnote{We keep the same notation for the pull-backs of the $\kappa$ classes via the structure map $\phi$. Also, we identify $\AA\left(\MM_{0,0}(\pi_\Lambda, L), \mathbb{Q}\right)$ as a subring of $\AA(\mathcal{X}_{\MM}^n, \mathbb{Q})$ via $\pi_{\MM}^{n*}$.}} to 
$$\frac{1}{24} \,\cdot\, \left(\w{\kappa}_{[L;1]} - \kappa_{[L;1]}\right) \ \in \ \mathsf{A}^1\left(\MM_{0,0}(\pi_\Lambda, L), \mathbb{Q}\right) \,. $$
Therefore,
\begin{equation} \label{th51}
\w{\mathcal{L}} - \frac{1}{24}\,\cdot\,\w{\kappa}_{[L;1]} \,=\, \mathcal{L} - \frac{1}{24}\,\cdot\,\kappa_{[L;1]} \ \in \ \mathsf{A}^1(\mathcal{X}_{\MM}, \mathbb{Q}) \,.
\end{equation}

Our strategy for proving Theorem \ref{dldl} is to export the WDVV relation 
via the morphisms
$$\MM_{0,4}\ \stackrel{\tau}{\longleftarrow} \ \MM_{0,4}(\pi_\Lambda, L) \ 
\stackrel{\epsilon_{\MM}^4}{\longrightarrow}\  
\mathcal{X}^4_{\MM}\, .$$
We deduce the following identity from the exported relation
\begin{equation} \label{wwwu}
\epsilon_{\MM*}^4 \tau^*(\mathsf{WDVV}) \,=\, 0 \ \in\ \mathsf{A}_{\mathsf{d}(\Lambda) + 3}(\mathcal{X}^4_{\MM}, \mathbb{Q})\,,
\end{equation}
where $\mathsf{d}(\Lambda) = 20 - \text{rank}(\Lambda)$ is the dimension of $\mathcal{M}_\Lambda$. 

\begin{proposition} \label{blab}
For $L\in \Lambda$ admissible,
\begin{equation*} \label{th52}
\w{\kappa}_{[L;1]} \cap \left[\MM_{0,0}(\pi_\Lambda, L)\right]^{\textup{red}} \,=\, \kappa_{[L;1]} \cap \left[\MM_{0,0}(\pi_\Lambda, L)\right]^{\textup{red}} \ \in \ \mathsf{A}_{\mathsf{d}(\Lambda) - 1}\left(\MM_{0,0}(\pi_\Lambda, L), \mathbb{Q}\right)\, .
\end{equation*}
\end{proposition}

Equation \eqref{th51} and 
Proposition \ref{blab} together yield
$$\w{\mathcal{L}}\cap [\mathcal{X}_\MM]^{\textup{red}}
\, =\, \mathcal{L} \cap [\mathcal{X}_\MM]^{\textup{red}}\ \in\  \mathsf{A}_{\mathsf{d}(\Lambda)+1}(\mathcal{X}_\MM,\mathbb{Q})\, ,$$
thus proving Theorem \ref{dldl}.

The exportation process is almost identical to the one in Section \ref{wwww}. However, since we work over $\MM_{0,0}(\pi_\Lambda, L)$ instead of $\mathcal{M}_\Lambda$, we do {\it not} require Proposition \ref{zzzr} (whose proof uses Theorem \ref{dldl}).

\subsection{Exportation}
We briefly describe the exportation \eqref{wwwu} of the WDVV relation with respect to the curve class $L$. As in Section \ref{wwww}, the outcome of $\epsilon_{\MM*}^4\tau^*(\mathsf{WDVV})$ consists of unsplit and split contributions:

\vspace{8pt}
\noindent \makebox[12pt][l]{$\bullet$}For the unsplit contributions, the difference is that one should replace $\mathcal{L}$ by the corresponding $\w{\mathcal{L}}$. Moreover, since we do {\it not} push-forward to~$\mathcal{X}^4_\Lambda$, there is no overall coefficient~$N_0(L)$.

\vspace{8pt}
\noindent \makebox[12pt][l]{$\bullet$}For the split contributions corresponding to the admissible curve class distributions 
$$L = L_1 + L_2 \,,$$
one again replaces $\mathcal{L}_i$ by the corresponding $\w{\mathcal{L}}_i$ and removes the coefficient $N_0(L_i)$. As before, the terms are either supported over proper Noether-Lefschetz divisors of $\mathcal{M}_\Lambda$, or multiplied by (the pull-back of) $-\lambda$.

\vspace{8pt}
\noindent We obtain the following analog of Theorem \ref{WDVV}.

\pagebreak

\begin{proposition}
For admissible $L\in \Lambda$,
exportation of the WDVV relation yields
\begin{multline} \label{dada}
\Big(\w{\mathcal{L}}_{(1)} \w{\mathcal{L}}_{(2)} \w{\mathcal{L}}_{(3)} \Delta_{(34)} + \w{\mathcal{L}}_{(1)}\w{\mathcal{L}}_{(3)} \w{\mathcal{L}}_{(4)} \Delta_{(12)} - \w{\mathcal{L}}_{(1)}\w{\mathcal{L}}_{(2)}\w{\mathcal{L}}_{(3)}\Delta_{(24)} \\
- \w{\mathcal{L}}_{(1)}\w{\mathcal{L}}_{(2)}\w{\mathcal{L}}_{(4)}\Delta_{(13)} + \ldots\Big) \cap [\mathcal{X}_\MM^4]^{\textup{red}}
\, = \, 0 \ \in \ \mathsf{A}_{\mathsf{d}(\Lambda) + 3}(\mathcal{X}_{\MM}^4, \mathbb{Q})\,,
\end{multline}
where the dots stand for (Gromov-Witten) tautological classes supported over proper Noether-Lefschetz divisors
of $\mathcal{M}_\Lambda$.
\end{proposition}

\noindent Here, the Gromov-Witten tautological classes on $\mathcal{X}^n_{\MM}$ are defined by replacing $\mathcal{L}$ by $\w{\mathcal{L}}$ in Section \ref{ttun}.

\subsection{Proof of Proposition \ref{blab}}
We distinguish two cases.

\vspace{8pt}
\noindent {\bf Case $\langle L, L\rangle_\Lambda \neq 0$.} \nopagebreak

\vspace{8pt}
First, we rewrite  \eqref{th51} as
$$\w{\kappa}_{[L;1]} - \kappa_{[L;1]} \,=\, 24 \cdot (\w{\mathcal{L}} - \mathcal{L}) \ \in \ \mathsf{A}^1(\mathcal{X}_{\MM}, \mathbb{Q}) \,.$$
By the same argument, we also have
$$\w{\kappa}_{[L^3;0]} - \kappa_{[L^3;0]} \,=\, 3 \langle L, L\rangle_\Lambda \cdot (\w{\mathcal{L}} - \mathcal{L}) \ \in \ \mathsf{A}^1(\mathcal{X}_{\MM}, \mathbb{Q}) \,.$$
By combining the above equations, we find
\begin{equation} \label{mir1}
\langle L, L\rangle_\Lambda \cdot \w{\kappa}_{[L;1]} - 8 \cdot \w{\kappa}_{[L^3;0]} \,=\, \langle L, L\rangle_\Lambda \cdot {\kappa}_{[L;1]} - 8 \cdot {\kappa}_{[L^3;0]} \ \in \ \mathsf{A}^1\left(\MM_{0,0}(\pi_\Lambda, L), \mathbb{Q}\right)\,.
\end{equation}

Next, we apply \eqref{dada} with respect to $L$ and insert $\Delta_{(12)}\Delta_{(34)} \in \mathsf{A}^4(\mathcal{X}_\MM^4,\mathbb{Q})$. The relation
$$\Delta_{(12)}\Delta_{(34)} \cap \epsilon^4_{\MM*}\tau^*(\mathsf{WDVV}) \,=\, 0 \ \in \ \mathsf{A}_{\mathsf{d}(\Lambda) - 1}(\mathcal{X}^4_\MM, \mathbb{Q})$$
pushes down via
$$\pi^4_\MM: \mathcal{X}_{\MM}^4 \rightarrow \MM_{0,0}(\pi_\Lambda, L)$$
to yield the result
\begin{multline} \label{mir2}
\Big(2\langle L, L\rangle_\Lambda \cdot \w{\kappa}_{[L;1]} - 2 \cdot \w{\kappa}_{[L^3;0]}\Big) \cap \left[\MM_{0,0}(\pi_\Lambda, L)\right]^{\text{red}} \\
\in \ \phi^*\,\mathsf{NL}^1(\mathcal{M}_\Lambda, \mathbb{Q}) \cap \left[\MM_{0,0}(\pi_\Lambda, L)\right]^{\text{red}}\,.
\end{multline}

Since $\langle L, L\rangle_\Lambda \neq 0$, a combination of \eqref{mir1} and \eqref{mir2} yields
$$\w{\kappa}_{[L;1]} \cap \left[\MM_{0,0}(\pi_\Lambda, L)\right]^{\text{red}} \ \in \ \phi^*\,\mathsf{A}^1(\mathcal{M}_\Lambda, \mathbb{Q}) \cap \left[\MM_{0,0}(\pi_\Lambda, L)\right]^{\text{red}}\,.$$
In other words, there is a divisor class ${D} \in \mathsf{A}^1(\mathcal{M}_\Lambda, \mathbb{Q})$ for
which
\begin{equation*}
\w{\kappa}_{[L;1]} \cap \left[\MM_{0,0}(\pi_\Lambda, L)\right]^{\text{red}} \,=\, \phi^*({D}) \cap \left[\MM_{0,0}(\pi_\Lambda, L)\right]^{\text{red}} \ \in \ \mathsf{A}_{\mathsf{d}(\Lambda) - 1}\left(\MM_{0,0}(\pi_\Lambda, L), \mathbb{Q}\right)\,.
\end{equation*}
Then, by the projection formula, we find
$$\phi_*\left(\w{\kappa}_{[L;1]} \cap \left[\MM_{0,0}(\pi_\Lambda, L)\right]^{\text{red}}\right) \,=\, N_0(L) \cdot \kappa_{[L;1]} \,=\, N_0(L) \cdot {D} \ \in \ \mathsf{A}^1(\mathcal{M}_\Lambda, \mathbb{Q}) \,.$$
Hence ${D} = \kappa_{[L;1]}$, 
which proves Proposition \ref{blab} in case
$\langle L, L\rangle_\Lambda \neq 0$.

\vspace{8pt}
\noindent {\bf Case $\langle L, L\rangle_\Lambda = 0$.} \nopagebreak

\vspace{8pt}
Let $H \in \Lambda$ be the quasi-polarization and let
$$\mathcal{H} \ \in \ \mathsf{A}^1(\mathcal{X}_{\MM}, \mathbb{Q})$$
be the pull-back of the class $\mathcal{H} \in \mathsf{A}^1(\mathcal{X}_\Lambda, \mathbb{Q})$. We define the $\kappa$ classes
$$\w{\kappa}_{[H^{a_1},L^{a_2};b]} \,=\, \pi_{\MM*}\left(\mathcal{H}^{a_1} \cdot \w{\mathcal{L}}^{a_2} \cdot c_2(\mathcal{T}_{\pi_{\MM}})^b\right) \ \in \ \mathsf{A}^{a_1 + a_2 + 2b - 2}\left(\MM_{0,0}(\pi_\Lambda, L), \mathbb{Q}\right) \,.$$

First, by the same argument used to prove \eqref{th51}, we have
$$\w{\kappa}_{[H,L^2;0]} - {\kappa}_{[H,L^2;0]} \,=\, 2\langle H, L\rangle_\Lambda \cdot (\w{\mathcal{L}} - \mathcal{L}) \ \in \ \mathsf{A}^1(\mathcal{X}_{\MM}, \mathbb{Q}) \,.$$
By combining the above equation with \eqref{th51}, we find
\begin{multline} \label{mir3}
\langle H, L\rangle_\Lambda \cdot \w{\kappa}_{[L;1]} - 12 \cdot \w{\kappa}_{[H,L^2;0]} \\ =\, \langle H, L\rangle_\Lambda \cdot {\kappa}_{[L;1]} - 12 \cdot {\kappa}_{[H,L^2;0]} \ \in \ \mathsf{A}^1\left(\MM_{0,0}(\pi_\Lambda, L), \mathbb{Q}\right) \,.
\end{multline}

Next, we apply \eqref{dada} with respect to $L$ and insert $\mathcal{H}_{(1)}\mathcal{H}_{(2)}\Delta_{(34)} \in \mathsf{A}^4(\mathcal{X}_\MM^4,\mathbb{Q})$. The relation
$$\mathcal{H}_{(1)}\mathcal{H}_{(2)}\Delta_{(34)} \cap \epsilon^4_{\MM*}\tau^*(\mathsf{WDVV}) \,=\, 0 \ \in\  \mathsf{A}_{\mathsf{d}(\Lambda) - 1}(\mathcal{X}^4_\MM, \mathbb{Q})$$
pushes down via $\pi^4_\MM$ to yield the result
\begin{multline} \label{mir4}
\Big(\langle H, L\rangle_\Lambda^2 \cdot \w{\kappa}_{[L;1]} - 2\langle H, L\rangle_\Lambda \cdot \w{\kappa}_{[H,L^2;0]}\Big) \cap \left[\MM_{0,0}(\pi_\Lambda, L)\right]^{\text{red}} \\
\in \ \phi^*\,\mathsf{NL}^1(\mathcal{M}_\Lambda, \mathbb{Q}) \cap \left[\MM_{0,0}(\pi_\Lambda, L)\right]^{\text{red}}\,.
\end{multline}

Since $\langle H, L\rangle_\Lambda \neq 0$ by the Hodge index theorem, a combination of \eqref{mir3} and \eqref{mir4} yields
$$\w{\kappa}_{[L;1]} \cap \left[\MM_{0,0}(\pi_\Lambda, L)\right]^{\text{red}} \ \in \ \phi^*\,\mathsf{A}^1(\mathcal{M}_\Lambda, \mathbb{Q}) \cap \left[\MM_{0,0}(\pi_\Lambda, L)\right]^{\text{red}}\,.$$
As in the previous case, we conclude
$$\w{\kappa}_{[L;1]} \cap \left[\MM_{0,0}(\pi_\Lambda, L)\right]^{\text{red}} \,=\, \kappa_{[L;1]} \cap \left[\MM_{0,0}(\pi_\Lambda, L)\right]^{\text{red}} \ \in \ \mathsf{A}_{\mathsf{d}(\Lambda) - 1}\left(\MM_{0,0}(\pi_\Lambda, L), \mathbb{Q}\right)\,.$$
The proof of Proposition \ref{blab} (and thus Theorem \ref{dldl}) is complete. \qed

\section{Exportation of Getzler's relation} \label{gggg}

\subsection{Exportation} 
Let $L \in \Lambda$ be an admissible class satisfying
$\langle L,L\rangle_\Lambda\geq 0$. 
Consider the
morphisms
$$\MM_{1,4}\ \stackrel{\tau}{\longleftarrow} \ \MM_{1,4}(\pi_\Lambda, L) \ 
\stackrel{\epsilon^4}{\longrightarrow}\  
\mathcal{X}^4_{\Lambda}\, .$$
Following the notation of Section \ref{conjs}, we export here Getzler's relation with respect to 
the curve class $L$,
\begin{equation}\label{exex1}
\epsilon_*^4 \tau^*(\mathsf{Getzler}) \,=\, 0 \ \in\ \mathsf{A}^5
(\mathcal{X}^4_\Lambda, \mathbb{Q})\, .
\end{equation}
We will compute $\epsilon_*^4 \tau^*(\mathsf{Getzler})$
by applying the splitting
axiom of Gromov-Witten theory to the 7 terms of Getzler's relation 
\eqref{getzler}.
The splitting axiom requires a distribution of the curve class
to each vertex of each graph appearing in \eqref{getzler}.

\subsection{Curve class distributions} 
To export Getzler's relation with respect to the curve class $L$, we will use the following properties for the graphs which arise:
\begin{itemize}
\item[(i)] Only distributions of admissible classes contribute.
\item[(ii)] A genus 1 vertex with valence{\footnote{The valence counts all
incident half-edges (both from edges and markings).}} 2 or a genus 0 vertex with valence at least 4 must carry a nonzero class.
\item[(iii)] A genus 1 vertex with valence 1 cannot be adjacent to a genus 0 vertex with a nonzero class.
\item[(iv)] A genus 1 vertex with valence 2 cannot be adjacent to two genus 0 vertices with nonzero classes.
\end{itemize}
Property (i) is a consequence of Propositions \ref{zzzr}, \ref{g1p1}, and \ref{g1g1}.
For Property (ii), the moduli of contracted 2-pointed genus 1 curve
produces a positive dimensional fiber of the push-forward to
$\mathcal{X}_\Lambda^4$ (and similarly for 
contracted 4-point genus 0 curves). Properties~(iii) and~(iv) are
 consequences of  positive dimensional fibers of
the push-forward to $\mathcal{X}_\Lambda^4$ obtained
from the elliptic component. We leave the elementary details
to the reader.

\subsection{Getzler's relation: unsplit contributions} \label{gusp}
We begin with the unsplit contributions. The strata appearing in Getzler's relation are ordered as in \eqref{getzler}.

\vspace{8pt}
\noindent {\bf Stratum 1.}
$$12\left[\ 
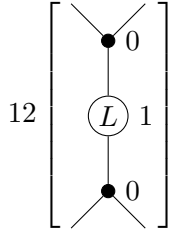
\begin{tikzpicture}[baseline={([yshift=-.5ex]current bounding box.center)}]
	\node[leg] (l3) at (0,3) {};
	\node[leg] (l4) at (1,3) {};
	\node[vertex] (v3) at (.5,2.5) [label=right:$0$] {};
	\node[circ] (v2) at (.5,1.5) [label=right:$1$] {$L$};
	\node[vertex] (v1) at (.5,.5) [label=right:$0$] {};
	\node[leg] (l1) at (0,0) {};
	\node[leg] (l2) at (1,0) {};
	\path
		(l3) edge (v3)
		(l4) edge (v3)
		(v3) edge (v2)
		(v2) edge (v1)
		(v1) edge (l1)
		(v1) edge (l2)
	;
\end{tikzpicture}\!\right]$$
\begin{multline*}
12 N_1(L) \cdot \Big(\mathcal{L}_{(1)}\Delta_{(12)}\Delta_{(34)} + \mathcal{L}_{(3)}\Delta_{(12)}\Delta_{(34)} + \mathcal{L}_{(1)}\Delta_{(13)}\Delta_{(24)} \\
+ \mathcal{L}_{(2)}\Delta_{(13)}\Delta_{(24)} + \mathcal{L}_{(1)}\Delta_{(14)}\Delta_{(23)} + \mathcal{L}_{(2)}\Delta_{(14)}\Delta_{(23)}\Big) \\
+ 12 N_1(L) \cdot Z(L) \Big(\Delta_{(12)}\Delta_{(34)} + \Delta_{(13)}\Delta_{(24)} + \Delta_{(14)}\Delta_{(23)}\Big)
\end{multline*}

\vspace{8pt}
\noindent By Property (ii), the genus 1 vertex must carry the curve class
$L$ in the unsplit case. The contribution is then calculated
using Propositions \ref{zzzr} and \ref{g1g1}.

\vspace{8pt}
\noindent {\bf Stratum 2.}
$$-4\left[\ 
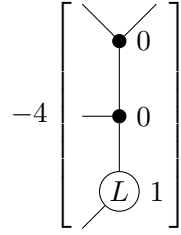
\begin{tikzpicture}[baseline={([yshift=-.5ex]current bounding box.center)}]
	\node[leg] (l3) at (0,3) {};
	\node[leg] (l4) at (1,3) {};
	\node[vertex] (v3) at (.5,2.5) [label=right:$0$] {};
	\node[leg] (l2) at (0,1.5) {};
	\node[vertex] (v2) at (.5,1.5) [label=right:$0$] {};
	\node[circ] (v1) at (.5,.5) [label=right:$1$] {$L$};
	\node[leg] (l1) at (0,0) {};
	\path
		(l3) edge (v3)
		(l4) edge (v3)
		(v3) edge (v2)
		(l2) edge (v2)
		(v2) edge (v1)
		(v1) edge (l1)
	;
\end{tikzpicture}\!\right]$$
\begin{multline*}
{-12}N_1(L) \cdot \Big(\mathcal{L}_{(1)}\Delta_{(234)} + \mathcal{L}_{(2)}\Delta_{(134)} + \mathcal{L}_{(3)}\Delta_{(124)} + \mathcal{L}_{(4)}\Delta_{(123)} \\
+ \mathcal{L}_{(1)}\Delta_{(123)} + \mathcal{L}_{(1)}\Delta_{(124)} + \mathcal{L}_{(1)}\Delta_{(134)} + \mathcal{L}_{(2)}\Delta_{(234)}\Big) \\
-12 N_1(L) \cdot Z(L) \Big(\Delta_{(123)} + \Delta_{(124)} + \Delta_{(134)} + \Delta_{(234)}\Big)
\end{multline*}

\vspace{8pt}
\noindent Again by Property (ii), the genus 1 vertex must carry the curve class
$L$ in the unsplit case. The contribution is then calculated
using Propositions \ref{zzzr} and \ref{g1g1}.

\vspace{8pt}
\noindent {\bf Stratum 3.} No contribution by Properties (ii) and (iii).

\vspace{8pt}
\noindent {\bf Stratum 4.}
$$6\left[\ 
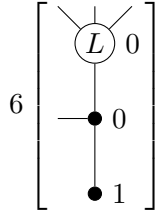
\begin{tikzpicture}[baseline={([yshift=-.3ex]current bounding box.center)}]
	\node[leg] (l2) at (0,2.5) {};
	\node[leg] (l3) at (.5,2.5) {};
	\node[leg] (l4) at (1,2.5) {};
	\node[circ] (v3) at (.5,2) [label=right:$0$] {$L$};
	\node[leg] (l1) at (0,1) {};
	\node[vertex] (v2) at (.5,1) [label=right:$0$] {};
	\node[vertex] (v1) at (.5,0) [label=right:$1$] {};
	\path
	    (l2) edge (v3)
		(l3) edge (v3)
		(l4) edge (v3)
		(v3) edge (v2)
		(l1) edge (v2)
		(v2) edge (v1)
	;
\end{tikzpicture}\!\right]$$
\vspace{0pt}
$$N_0(L) \cdot \lambda \mathcal{L}_{(1)}\mathcal{L}_{(2)}\mathcal{L}_{(3)}\mathcal{L}_{(4)}$$

\vspace{8pt}
\noindent The genus 0 vertex of valence 4 must carry the curve class $L$ in the
unsplit case.
The contracted genus 1 vertex contributes the virtual class
\begin{equation}\label{lwwl}
\epsilon_*[\overline{\mathsf{M}}_{1,1}(\pi_\Lambda,0)]^{\text{vir}} \,=\, \frac{1}{24}\,\cdot\,\lambda
\ \in \ \mathsf{A}^1(\mathcal{X}^1_\Lambda,\mathbb{Q})\, .
\end{equation}
The coefficient 6 together with the 4 graphs which occur cancel
the 24 in the denominator of \eqref{lwwl}. Proposition \ref{zzzr}
is then applied to the  genus 0 vertex of valence 4.

\vspace{8pt}
\noindent {\bf Stratum 5.} No contribution by Property (ii) since there
are two genus 0 vertices of valence 4.

\vspace{8pt}
\noindent {\bf Stratum 6.}
$$\left[
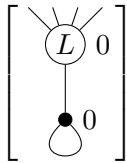
\begin{tikzpicture}[baseline={([yshift=-.3ex]current bounding box.center)}]
	\node[leg] (l1) at (0,1.5) {};
	\node[leg] (l2) at (.33,1.5) {};
	\node[leg] (l3) at (.67,1.5) {};
	\node[leg] (l4) at (1,1.5) {};
	\node[circ] (v2) at (.5,1) [label=right:$0$] {$L$};
	\node[vertex] (v1) at (.5,0) [label=right:$0$] {};
	\path
        (l1) edge (v2)
        (l2) edge (v2)
		(l3) edge (v2)
		(l4) edge (v2)
		(v2) edge (v1)
		(v1) edge[in=-135,out=-45,loop] (v1)
	;
\end{tikzpicture}\!\right]$$
$$\frac{1}{2}N_0(L) \cdot \kappa_{[L;1]} \mathcal{L}_{(1)}\mathcal{L}_{(2)}\mathcal{L}_{(3)}\mathcal{L}_{(4)}$$

\vspace{8pt}
\noindent The genus 0 vertex of valence 4 must carry the curve class $L$ in the unsplit case.
Proposition \ref{zzzr}
is applied to the  genus 0 vertex of valence 4.
The self-edge of the contracted  genus 0 vertex yields
a factor of $c_2(\mathcal{T}_{\pi_{\Lambda}})$. The contribution
of the contracted genus 0 vertex is
$$\frac{1}{2}\,\cdot\,\kappa_{[L;1]}$$
where the factor of $\frac{1}{2}$ is included since the self-edge
is not oriented.

\vspace{8pt}
\noindent {\bf Stratum 7.} No contribution
by Property (ii) since there
are two genus 0 vertices of valence 4.

\vspace{8pt}
We have already seen that $\lambda$ is expressible
in term of the Noether-Lefschetz divisors of~$\mathcal{M}_{\Lambda}$.
Since we will later express $Z(L)$ and $\kappa_{[L;1]}$
in terms of the Noether-Lefschetz divisors of $\mathcal{M}_{\Lambda}$,
the principal terms in the above analysis only occur in
Strata 1 and~2. 
The principal parts of Strata 1 and 2 (divided{\footnote{The admissibility
of $L$ together with
condition $\langle L,L\rangle_\Lambda \geq 0$ implies $N_1(L)\neq 0$
by  Proposition \ref{trtrtr}.}}
by $12N_1(L)$) exactly 
constitute the principal part of Theorem \ref{ggg}.

\subsection{Getzler's relation: split contributions} \label{gsp}
The split contributions are obtained
from non-trivial curve class distributions to the vertices.
By Property (i), we need only consider distributions of
admissible classes.

\vspace{8pt}
\noindent {\bf Case A.} The class $L$ is
divided into two nonzero parts
$$L = L_1 + L_2\,.$$ 
Let $\ww\Lambda$ be the saturation of the span of $L_1$, $L_2$, and $\Lambda$.

\begin{enumerate}
\item[$\bullet$] If $\text{rank}(\ww{\Lambda}) = \text{rank}(\Lambda) + 1$, the 
contributions are pushed forward from $\mathcal{X}^4_{\ww{\Lambda}}$ via 
the map $\mathcal{X}^4_{\ww{\Lambda}} \to \mathcal{X}^4_\Lambda$.

\item[$\bullet$] If $\ww{\Lambda} = \Lambda$, the contributions are multiplied by $-\lambda$.
\end{enumerate}

\noindent With the above rules, the formulas below address both the 
$\text{rank}(\ww{\Lambda}) = \text{rank}(\Lambda) + 1$ 
and the $\text{rank}(\ww{\Lambda}) = \text{rank}(\Lambda)$
cases simultaneously.

\vspace{8pt}
\noindent {\bf Stratum 1.}
$$12\left[\ \begin{tikzpicture}[baseline={([yshift=-.5ex]current bounding box.center)}]
	\node[leg] (l3) at (0,3) {};
	\node[leg] (l4) at (1,3) {};
	\node[vertex] (v3) at (.5,2.5) [label=right:$0$] {};
	\node[circ] (v2) at (.5,1.5) [label=right:$1$] {$L_1$};
	\node[circ] (v1) at (.5,.5) [label=right:$0$] {$L_2$};
	\node[leg] (l1) at (0,0) {};
	\node[leg] (l2) at (1,0) {};
	\path
		(l3) edge (v3)
		(l4) edge (v3)
		(v3) edge (v2)
		(v2) edge (v1)
		(v1) edge (l1)
		(v1) edge (l2)
	;
\end{tikzpicture}\!\right]$$
\begin{multline*}
12N_1(L_1)N_0(L_2)\langle L_1, L_2\rangle_{\ww{\Lambda}} \cdot \Big(\mathcal{L}_{2,(1)}\mathcal{L}_{2,(2)}\Delta_{(34)} + \mathcal{L}_{2,(3)}\mathcal{L}_{2,(4)}\Delta_{(12)} \\
+ \mathcal{L}_{2,(1)}\mathcal{L}_{2,(3)}\Delta_{(24)} + \mathcal{L}_{2,(2)}\mathcal{L}_{2,(4)}\Delta_{(13)} + \mathcal{L}_{2,(1)}\mathcal{L}_{2,(4)}\Delta_{(23)} + \mathcal{L}_{2,(2)}\mathcal{L}_{2,(3)}\Delta_{(14)}\Big)
\end{multline*}

\vspace{8pt}
\noindent By Property (ii), the genus 1 vertex must carry a nonzero curve class.
The contribution is calculated using Propositions \ref{zzzr} and \ref{g1g1}.

\vspace{8pt}
\noindent {\bf Stratum 2.}
$$-4\left[\ 
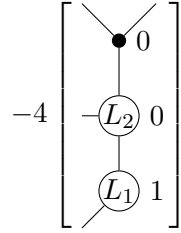
\begin{tikzpicture}[baseline={([yshift=-.5ex]current bounding box.center)}]
	\node[leg] (l3) at (0,3) {};
	\node[leg] (l4) at (1,3) {};
	\node[vertex] (v3) at (.5,2.5) [label=right:$0$] {};
	\node[leg] (l2) at (0,1.5) {};
	\node[circ] (v2) at (.5,1.5) [label=right:$0$] {$L_2$};
	\node[circ] (v1) at (.5,.5) [label=right:$1$] {$L_1$};
	\node[leg] (l1) at (0,0) {};
	\path
		(l3) edge (v3)
		(l4) edge (v3)
		(v3) edge (v2)
		(l2) edge (v2)
		(v2) edge (v1)
		(v1) edge (l1)
	;
\end{tikzpicture}\!\right]$$
\begin{multline*}
{-4}N_1(L_1)N_0(L_2)\langle L_1, L_2\rangle_{\ww{\Lambda}} \cdot \Big(\mathcal{L}_{2,(1)}\mathcal{L}_{2,(2)}\Delta_{(23)} + \mathcal{L}_{2,(1)}\mathcal{L}_{2,(2)}\Delta_{(24)} + \mathcal{L}_{2,(1)}\mathcal{L}_{2,(3)}\Delta_{(34)} \\
+ \mathcal{L}_{2,(1)}\mathcal{L}_{2,(2)}\Delta_{(13)} + \mathcal{L}_{2,(1)}\mathcal{L}_{2,(2)}\Delta_{(14)} + \mathcal{L}_{2,(2)}\mathcal{L}_{2,(3)}\Delta_{(34)} \\
+ \mathcal{L}_{2,(1)}\mathcal{L}_{2,(3)}\Delta_{(12)} + \mathcal{L}_{2,(1)}\mathcal{L}_{2,(3)}\Delta_{(14)} + \mathcal{L}_{2,(2)}\mathcal{L}_{2,(3)}\Delta_{(24)} \\
+ \mathcal{L}_{2,(1)}\mathcal{L}_{2,(4)}\Delta_{(12)} + \mathcal{L}_{2,(1)}\mathcal{L}_{2,(4)}\Delta_{(13)} + \mathcal{L}_{2,(2)}\mathcal{L}_{2,(4)}\Delta_{(23)}\Big)
\end{multline*}

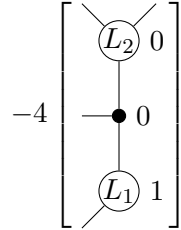
$$-4\left[\ \begin{tikzpicture}[baseline={([yshift=-.5ex]current bounding box.center)}]
	\node[leg] (l3) at (0,3) {};
	\node[leg] (l4) at (1,3) {};
	\node[circ] (v3) at (.5,2.5) [label=right:$0$] {$L_2$};
	\node[leg] (l2) at (0,1.5) {};
	\node[vertex] (v2) at (.5,1.5) [label=right:$0$] {};
	\node[circ] (v1) at (.5,.5) [label=right:$1$] {$L_1$};
	\node[leg] (l1) at (0,0) {};
	\path
		(l3) edge (v3)
		(l4) edge (v3)
		(v3) edge (v2)
		(l2) edge (v2)
		(v2) edge (v1)
		(v1) edge (l1)
	;
\end{tikzpicture}\!\right]$$
\begin{multline*}
{-12}N_1(L_1)N_0(L_2) \cdot \Big(\mathcal{L}_{1,(1)}\mathcal{L}_{2,(2)}\mathcal{L}_{2,(3)}\mathcal{L}_{2,(4)} + \mathcal{L}_{1,(2)}\mathcal{L}_{2,(1)}\mathcal{L}_{2,(3)}\mathcal{L}_{2,(4)} \\
+ \mathcal{L}_{1,(3)}\mathcal{L}_{2,(1)}\mathcal{L}_{2,(2)}\mathcal{L}_{2,(4)} + \mathcal{L}_{1,(4)}\mathcal{L}_{2,(1)}\mathcal{L}_{2,(2)}\mathcal{L}_{2,(3)}\Big) \\
- 4N_1(L_1)N_0(L_2) \cdot \Big(\mathcal{L}_{1,(1)}\mathcal{L}_{2,(1)}\mathcal{L}_{2,(2)}\mathcal{L}_{2,(3)} + \mathcal{L}_{1,(1)}\mathcal{L}_{2,(1)}\mathcal{L}_{2,(2)}\mathcal{L}_{2,(4)} + \mathcal{L}_{1,(1)}\mathcal{L}_{2,(1)}\mathcal{L}_{2,(3)}\mathcal{L}_{2,(4)} \\
+ \mathcal{L}_{1,(2)}\mathcal{L}_{2,(1)}\mathcal{L}_{2,(2)}\mathcal{L}_{2,(3)} + \mathcal{L}_{1,(2)}\mathcal{L}_{2,(1)}\mathcal{L}_{2,(2)}\mathcal{L}_{2,(4)} + \mathcal{L}_{1,(2)}\mathcal{L}_{2,(2)}\mathcal{L}_{2,(3)}\mathcal{L}_{2,(4)} \\
+ \mathcal{L}_{1,(3)}\mathcal{L}_{2,(1)}\mathcal{L}_{2,(2)}\mathcal{L}_{2,(3)} + \mathcal{L}_{1,(3)}\mathcal{L}_{2,(1)}\mathcal{L}_{2,(3)}\mathcal{L}_{2,(4)} + \mathcal{L}_{1,(3)}\mathcal{L}_{2,(2)}\mathcal{L}_{2,(3)}\mathcal{L}_{2,(4)} \\
\phantom{\Big(} + \mathcal{L}_{1,(4)}\mathcal{L}_{2,(1)}\mathcal{L}_{2,(2)}\mathcal{L}_{2,(3)} + \mathcal{L}_{1,(4)}\mathcal{L}_{2,(1)}\mathcal{L}_{2,(3)}\mathcal{L}_{2,(4)} + \mathcal{L}_{1,(4)}\mathcal{L}_{2,(2)}\mathcal{L}_{2,(3)}\mathcal{L}_{2,(4)}\Big) \\
- 12N_1(L_1)N_0(L_2) \cdot Z(L_1) \Big(\mathcal{L}_{2,(1)}\mathcal{L}_{2,(2)}\mathcal{L}_{2,(3)} + \mathcal{L}_{2,(1)}\mathcal{L}_{2,(2)}\mathcal{L}_{2,(4)} \\
+ \mathcal{L}_{2,(1)}\mathcal{L}_{2,(3)}\mathcal{L}_{2,(4)} + \mathcal{L}_{2,(2)}\mathcal{L}_{2,(3)}\mathcal{L}_{2,(4)}\Big) \,
\end{multline*}

\vspace{8pt}
\noindent By Property (ii), the genus 1 vertex must carry a nonzero curve class.
There are two possibilities for the distribution.
Both contributions are calculated using Propositions \ref{zzzr} and \ref{g1g1}.

\vspace{8pt}
\noindent {\bf Stratum 3.} 
No contribution by Properties (ii) and (iii).

\vspace{8pt}
\noindent {\bf Stratum 4.}
$$6\left[\ 
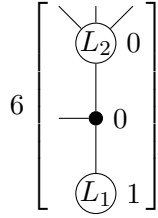
\begin{tikzpicture}[baseline={([yshift=-.3ex]current bounding box.center)}]
	\node[leg] (l2) at (0,2.5) {};
	\node[leg] (l3) at (.5,2.5) {};
	\node[leg] (l4) at (1,2.5) {};
	\node[circ] (v3) at (.5,2) [label=right:$0$] {$L_2$};
	\node[leg] (l1) at (0,1) {};
	\node[vertex] (v2) at (.5,1) [label=right:$0$] {};
	\node[circ] (v1) at (.5,0) [label=right:$1$] {$L_1$};
	\path
	    (l2) edge (v3)
		(l3) edge (v3)
		(l4) edge (v3)
		(v3) edge (v2)
		(l1) edge (v2)
		(v2) edge (v1)
	;
\end{tikzpicture}\!\right]$$
\vspace{0pt}
$$24N_1(L_1)N_0(L_2) \cdot \mathcal{L}_{2,(1)}\mathcal{L}_{2,(2)}\mathcal{L}_{2,(3)}\mathcal{L}_{2,(4)}$$

\vspace{8pt}
\noindent By Property (iii), the genus 0 vertex in the middle can not carry a
nonzero curve class. The contribution is calculated using Propositions \ref{zzzr} and \ref{g1p1}.

\vspace{8pt}
\noindent {\bf Stratum 5.}
$$\left[\ 
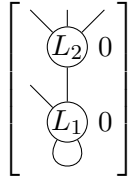
\begin{tikzpicture}[baseline={([yshift=-.3ex]current bounding box.center)}]
	\node[leg] (l2) at (0,1.5) {};
	\node[leg] (l3) at (.5,1.5) {};
	\node[leg] (l4) at (1,1.5) {};
	\node[circ] (v2) at (.5,1) [label=right:$0$] {$L_2$};
	\node[leg] (l1) at (0,.5) {};
	\node[circ] (v1) at (.5,0) [label=right:$0$] {$L_1$};
	\path
        (l2) edge (v2)
		(l3) edge (v2)
		(l4) edge (v2)
		(v2) edge (v1)
		(l1) edge (v1)
		(v1) edge[in=-120,out=-60,loopcirc] (v1)
	;
\end{tikzpicture}\!\right]$$
\begin{multline*}
\frac{1}{2} N_0(L_1)N_0(L_2)\langle L_1, L_1\rangle_{\ww{\Lambda}}\langle L_1, L_2\rangle_{\ww{\Lambda}} \cdot \Big(\mathcal{L}_{1,(1)}\mathcal{L}_{2,(2)}\mathcal{L}_{2,(3)}\mathcal{L}_{2,(4)} + \mathcal{L}_{1,(2)}\mathcal{L}_{2,(1)}\mathcal{L}_{2,(3)}\mathcal{L}_{2,(4)} \\
+ \mathcal{L}_{1,(3)}\mathcal{L}_{2,(1)}\mathcal{L}_{2,(2)}\mathcal{L}_{2,(4)} + \mathcal{L}_{1,(4)}\mathcal{L}_{2,(1)}\mathcal{L}_{2,(2)}\mathcal{L}_{2,(3)}\Big)
\end{multline*}

\vspace{8pt}
\noindent The factor $\frac{1}{2} \langle L_1, L_1\rangle_{\ww{\Lambda}}$
is obtained from the self-edge. The contribution is calculated
using Proposition \ref{zzzr}.

\vspace{8pt}
\noindent {\bf Stratum 6.}
$$\left[\ 
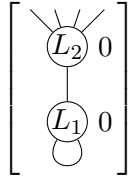
\begin{tikzpicture}[baseline={([yshift=-.3ex]current bounding box.center)}]
	\node[leg] (l1) at (0,1.5) {};
	\node[leg] (l2) at (.33,1.5) {};
	\node[leg] (l3) at (.67,1.5) {};
	\node[leg] (l4) at (1,1.5) {};
	\node[circ] (v2) at (.5,1) [label=right:$0$] {$L_2$};
	\node[circ] (v1) at (.5,0) [label=right:$0$] {$L_1$};
	\path
        (l1) edge (v2)
        (l2) edge (v2)
		(l3) edge (v2)
		(l4) edge (v2)
		(v2) edge (v1)
		(v1) edge[in=-120,out=-60,loopcirc] (v1)
	;
\end{tikzpicture}\!\right]$$
$$\frac{1}{2}N_0(L_1)N_0(L_2)\langle L_1, L_1\rangle_{\ww{\Lambda}}\langle L_1, L_2\rangle_{\ww{\Lambda}} \cdot \mathcal{L}_{2,(1)}\mathcal{L}_{2,(2)}\mathcal{L}_{2,(3)}\mathcal{L}_{2,(4)}$$

\vspace{8pt}
\noindent The factor $\frac{1}{2} \langle L_1, L_1\rangle_{\ww{\Lambda}}$
is obtained from the self-edge. The contribution is calculated
using Proposition \ref{zzzr}.

\vspace{8pt}
\noindent {\bf Stratum 7.}
$$-2\left[\ 
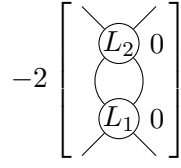
\begin{tikzpicture}[baseline={([yshift=-.5ex]current bounding box.center)}]
	\node[leg] (l3) at (0,2) {};
	\node[leg] (l4) at (1,2) {};
	\node[circ] (v2) at (.5,1.5) [label=right:$0$] {$L_2$};
	\node[circ] (v1) at (.5,.5) [label=right:$0$] {$L_1$};
	\node[leg] (l1) at (0,0) {};
	\node[leg] (l2) at (1,0) {};
	\path
		(l3) edge (v2)
		(l4) edge (v2)
		(v2) edge[bend left=45] (v1)
		(v2) edge[bend right=45] (v1)
		(v1) edge (l1)
		(v1) edge (l2)
	;
\end{tikzpicture}\!\right]$$
\begin{multline*}
{-}N_0(L_1)N_0(L_2)\langle L_1, L_2\rangle_{\ww{\Lambda}}^2 \cdot \Big(\mathcal{L}_{1,(1)}\mathcal{L}_{1,(2)}\mathcal{L}_{2,(3)}\mathcal{L}_{2,(4)} + \mathcal{L}_{2,(1)}\mathcal{L}_{2,(2)}\mathcal{L}_{1,(3)}\mathcal{L}_{1,(4)} \\
+ \mathcal{L}_{1,(1)}\mathcal{L}_{1,(3)}\mathcal{L}_{2,(2)}\mathcal{L}_{2,(4)} + \mathcal{L}_{2,(1)}\mathcal{L}_{2,(3)}\mathcal{L}_{1,(2)}\mathcal{L}_{1,(4)} \\
+ \mathcal{L}_{1,(1)}\mathcal{L}_{1,(4)}\mathcal{L}_{2,(2)}\mathcal{L}_{2,(3)} + \mathcal{L}_{2,(1)}\mathcal{L}_{2,(4)}\mathcal{L}_{1,(2)}\mathcal{L}_{1,(3)}\Big)
\end{multline*}

\vspace{8pt}
\noindent The factor $-2 \left(\frac{1}{2} \langle L_1, L_2\rangle^2_{\ww{\Lambda}}\right)$
is obtained from two middle edges (the $\frac{1}{2}$ comes
from the symmetry of the graph). The contribution is calculated
using Proposition \ref{zzzr}.

\vspace{8pt}
\noindent {\bf Case B.} The class $L$ is
divided into three nonzero parts
$$L = L_1 + L_2+L_3\,.$$ 
Let $\ww\Lambda$ be the saturation of the span of $L_1$, $L_2$, $L_3$, and $\Lambda$.
By Properties (ii)-(iv),
only Stratum 2 contributes.

\begin{enumerate}
\item[$\bullet$] If $\text{rank}(\ww{\Lambda}) = \text{rank}(\Lambda) + 2$, the
contributions are pushed forward from $\mathcal{X}^4_{\ww{\Lambda}}$ via 
the map $\mathcal{X}^4_{\ww{\Lambda}} \to \mathcal{X}^4_\Lambda$.

\item[$\bullet$] If $\text{rank}(\ww{\Lambda}) = \text{rank}(\Lambda) + 1$, 
the contributions are pushed forward from $\mathcal{X}^4_{\ww{\Lambda}}$ via 
the map $\mathcal{X}^4_{\ww{\Lambda}} \to \mathcal{X}^4_\Lambda$ {\it and}  multiplied by $-\lambda$.

\item[$\bullet$] If $\ww{\Lambda} = \Lambda$, the contributions
 are multiplied by $(-\lambda)^2$.
\end{enumerate}

\noindent With the above rules, the formula below addresses all three cases 
$$\text{rank}(\ww{\Lambda}) = \text{rank}(\Lambda) + 2\, ,\ \  
 \text{rank}(\ww{\Lambda}) = \text{rank}(\Lambda) + 1\, , \ \
\text{rank}(\ww{\Lambda}) = \text{rank}(\Lambda)$$
simultaneously.

\vspace{8pt}
\noindent {\bf Stratum 2.}
$$-4\left[\ \begin{tikzpicture}[baseline={([yshift=-.5ex]current bounding box.center)}]
	\node[leg] (l3) at (0,3) {};
	\node[leg] (l4) at (1,3) {};
	\node[circ] (v3) at (.5,2.5) [label=right:$0$] {$L_3$};
	\node[leg] (l2) at (0,1.5) {};
	\node[circ] (v2) at (.5,1.5) [label=right:$0$] {$L_2$};
	\node[circ] (v1) at (.5,.5) [label=right:$1$] {$L_1$};
	\node[leg] (l1) at (0,0) {};
	\path
		(l3) edge (v3)
		(l4) edge (v3)
		(v3) edge (v2)
		(l2) edge (v2)
		(v2) edge (v1)
		(v1) edge (l1)
	;
\end{tikzpicture}\!\right]$$
\begin{multline*}
{-4}N_1(L_1)N_0(L_2)N_0(L_3)\langle L_1, L_2\rangle_{\ww{\Lambda}}\langle L_2, L_3\rangle_{\ww{\Lambda}} \cdot \Big(\mathcal{L}_{2,(1)}\mathcal{L}_{3,(2)}\mathcal{L}_{3,(3)} + \mathcal{L}_{2,(1)}\mathcal{L}_{3,(2)}\mathcal{L}_{3,(4)} \\
+ \mathcal{L}_{2,(1)}\mathcal{L}_{3,(3)}\mathcal{L}_{3,(4)} + \mathcal{L}_{2,(2)}\mathcal{L}_{3,(1)}\mathcal{L}_{3,(3)} + \mathcal{L}_{2,(2)}\mathcal{L}_{3,(1)}\mathcal{L}_{3,(4)} + \mathcal{L}_{2,(2)}\mathcal{L}_{3,(3)}\mathcal{L}_{3,(4)} \\
+ \mathcal{L}_{2,(3)}\mathcal{L}_{3,(1)}\mathcal{L}_{3,(2)} + \mathcal{L}_{2,(3)}\mathcal{L}_{3,(1)}\mathcal{L}_{3,(4)} + \mathcal{L}_{2,(3)}\mathcal{L}_{3,(2)}\mathcal{L}_{3,(4)} \\
+ \mathcal{L}_{2,(4)}\mathcal{L}_{3,(1)}\mathcal{L}_{3,(2)} + \mathcal{L}_{2,(4)}\mathcal{L}_{3,(1)}\mathcal{L}_{3,(3)} + \mathcal{L}_{2,(4)}\mathcal{L}_{3,(2)}\mathcal{L}_{3,(3)}\Big) 
\end{multline*}

\vspace{8pt}
\noindent The contribution is calculated using Propositions \ref{zzzr} and \ref{g1g1}.

\subsection{Proof of Theorem \ref{ggg}}
The complete exported relation \eqref{exex1} is obtained by adding all the 
unsplit contributions of Section \ref{gusp} to all the 
split contributions of Section \ref{gsp}.
Using the Noether-Lefschetz support{\footnote{To be proven
in Section \ref{dvdv}.}} of
$$\lambda\, , \ \ \kappa_{[L;1]}\, , \ \ Z(L)$$
the only principal contributions are unsplit and
obtained from Strata 1 and 2. For the formula
of Theorem \ref{ggg}, we normalize the relation
by dividing by $12N_1(L)$.
\qed

\subsection{Higher genus relations}
In genus 2, there is a basic relation
among tautological classes in codimension 2 
on $\overline{\mathcal{M}}_{2,3}$, see \cite{BelP}. However, to export in genus 2, we
would first have to prove genus 2 analogues of
the push-forward results in genus 0 and
1 of Section
\ref{expo}. To build a theory which allows
the exportation of all the known tautological
relations{\footnote{For a
survey of Pixton's relations, see \cite{PSLC}.}}
on the moduli space of curves to
the moduli space of $K3$ surfaces is
an interesting direction of research.
Fortunately, to prove the Noether-Lefschetz
generation of Theorem \ref{dxxd}, only the
relations in genus 0 and 1 are required.

\section{Noether-Lefschetz generation} \label{pfpf}

\subsection{Overview} We present here the proof of Theorem \ref{dxxd}: the 
strict tautological ring is generated
by Noether-Lefschetz loci,
$$\NL(\mathcal{M}_{\Lambda}) = \mathsf{R}^\star(\mathcal{M}_{\Lambda})\, .$$
We will use the exported WDVV relation $(\dag)$ of Theorem \ref{wdvv},
the exported Getzler's relation $(\ddag)$ of Theorem \ref{ggg},
the diagonal decomposition $(\ddag')$ of Corollary
\ref{bvdiag},
and an induction on codimension.

For $(\ddag)$, we will require not only the principal terms which
appear in the statement of Theorem \ref{ggg}, but the entire
formula proven in Section \ref{gggg}. In particular, for $(\ddag)$
we will {\it not} divide by the factor $12N_1(L)$.

\subsection{Codimension $1$} \label{dvdv}
The base of the induction on codimension consists of all of the {\it divisorial} $\kappa$ classes: 
\begin{equation}\label{kakaka}
\kappa_{[L^3;0]}\,, \ \kappa_{[L;1]}\,, \ \kappa_{[L_1^2,L_2;0]}\,, \ \kappa_{[L_1,L_2,L_3;0]}\ \in \ \mathsf{R}^1(\mathcal{M}_{\Lambda})\, ,
\end{equation}
for $L, L_1, L_2, L_3 \in \Lambda$ admissible. 
Our first goal is to prove  the divisorial $\kappa$ classes \eqref{kakaka}
are expressible in terms of Noether-Lefschetz divisors in $\mathcal{M}_{\Lambda}$. In addition, we will determine the divisor $Z(L)$ defined in Proposition \ref{g1g1} for all $L \in \Lambda$ admissible and $\langle L, L\rangle_\Lambda \geq 0$.

Let $L, L_1, L_2, L_3 \in \Lambda$ be admissible, and let $H \in \Lambda$ be the quasi-polarization with 
$$\langle H, H\rangle_\Lambda = 2\ell > 0 \,.$$

\vspace{8pt}
\noindent {\bf Case A.} $\kappa_{[L^3;0]}$, $\kappa_{[L;1]}$, and $Z(L)$ for $\langle L, L\rangle_\Lambda > 0$. \nopagebreak

\vspace{8pt}
\noindent \makebox[12pt][l]{$\bullet$}We apply ($\dag$) with respect to $L$ and insert $\Delta_{(12)}\Delta_{(34)}\in \mathsf{R}^4(\mathcal{X}^4_\Lambda)$. The relation
$$\epsilon^4_*\tau^*(\mathsf{WDVV}) \cup \Delta_{(12)}\Delta_{(34)} \,=\, 0 \ \in\  \mathsf{R}^9(\mathcal{X}^4_\Lambda)$$
pushes down via
$$\pi^4_\Lambda: \mathcal{X}^4_\Lambda \rightarrow \mathcal{M}_\Lambda$$
to yield the result
\begin{equation} \label{wd1}
2\langle L, L\rangle_\Lambda \cdot \kappa_{[L;1]} - 2 \cdot \kappa_{[L^3;0]} \ \in \ \mathsf{NL}^1(\mathcal{M}_\Lambda) \,.
\end{equation}

\vspace{8pt}
\noindent \makebox[12pt][l]{$\bullet$}We apply ($\ddag$) with respect to $L$ and insert $\mathcal{L}_{(1)}\mathcal{L}_{(2)}\mathcal{L}_{(3)}\mathcal{L}_{(4)}\in \mathsf{R}^4(\mathcal{X}^4_\Lambda)$.
The relation
$$\epsilon^4_*\tau^*(\mathsf{Getzler}) \cup \mathcal{L}_{(1)}\mathcal{L}_{(2)}\mathcal{L}_{(3)}\mathcal{L}_{(4)} \,=\, 0 \ \in\  \mathsf{R}^9(\mathcal{X}^4_\Lambda)$$
pushes down via $\pi^4_\Lambda$
to yield the result
\begin{multline*}
72 N_1(L)\langle L, L\rangle_\Lambda \cdot \kappa_{[L^3;0]} + 36 N_1(L) \langle L, L\rangle_\Lambda^2 \cdot Z(L) \\
- 48N_1(L)\langle L, L\rangle_\Lambda \cdot \kappa_{[L^3;0]} + \frac{1}{2}N_0(L) \langle L, L\rangle_\Lambda^4 \cdot \kappa_{[L;1]} \ \in \ \mathsf{NL}^1(\mathcal{M}_\Lambda) \,.
\end{multline*}
The divisors $Z(L)$ and $\kappa_{[L;1]}$ are obtained from 
the unsplit contributions of Strata 1, 2, and 6.
After combining terms, we find
\begin{equation}\label{get1}
24 N_1(L) \cdot \kappa_{[L^3;0]} + \frac{1}{2}N_0(L) \langle L, L\rangle_\Lambda^3 \cdot \kappa_{[L;1]} + 36 N_1(L) \langle L, L\rangle_\Lambda \cdot Z(L) \ \in \ \mathsf{NL}^1(\mathcal{M}_\Lambda) \,.
\end{equation}

\vspace{8pt}
\noindent \makebox[12pt][l]{$\bullet$}We apply ($\ddag$) with respect to $L$ and insert
$\mathcal{L}_{(1)}\mathcal{L}_{(2)}\Delta_{(34)}\in \mathsf{R}^4(\mathcal{X}^4_\Lambda)$.
After push-down via $\pi^4_\Lambda$ to $\mathcal{M}_\Lambda$, we obtain
\begin{multline*}
288N_1(L) \cdot \kappa_{[L^3;0]} + 12N_1(L)\langle L, L\rangle_\Lambda \cdot \kappa_{[L;1]} + 48N_1(L)\cdot \kappa_{[L^3;0]} \\
+ 288N_1(L)\langle L, L\rangle_\Lambda \cdot Z(L)  + 24N_1(L)\langle L, L\rangle_\Lambda \cdot Z(L) \\
- 24N_1(L)\langle L, L\rangle_\Lambda \cdot \kappa_{[L;1]} - 24N_1(L) \cdot \kappa_{[L^3;0]} - 24N_1(L) \cdot \kappa_{[L^3;0]} \\
- 24N_1(L)\langle L, L\rangle_\Lambda \cdot Z(L) + \frac{1}{2}N_0(L)\langle L, L\rangle_\Lambda^3 \cdot \kappa_{[L;1]} \ \in \ \mathsf{NL}^1(\mathcal{M}_\Lambda) \,.
\end{multline*}
After combining terms, we find
\begin{multline} \label{get2}
288N_1(L) \cdot \kappa_{[L^3;0]} - \Big( 12N_1(L)\langle L, L\rangle_\Lambda -\frac{1}{2} N_0(L)\langle L, L\rangle_\Lambda^3\Big) \cdot \kappa_{[L;1]} \\
+ 288N_1(L)\langle L, L\rangle_\Lambda \cdot Z(L) \ \in \ \mathsf{NL}^1(\mathcal{M}_\Lambda) \,.
\end{multline}

\vspace{8pt}
\noindent \makebox[12pt][l]{$\bullet$}We apply ($\ddag$) with respect to $L$ and insert $\Delta_{(12)}\Delta_{(34)}
\in \mathsf{R}^4(\mathcal{X}^4_\Lambda)$.
After push-down via $\pi^4_\Lambda$ to $\mathcal{M}_\Lambda$, we obtain
\begin{multline*}
576N_1(L) \cdot \kappa_{[L;1]} + 48N_1(L) \cdot \kappa_{[L;1]} + 6912N_1(L) \cdot Z(L) + 576N_1(L) \cdot Z(L) \\
- 48N_1(L) \cdot \kappa_{[L;1]} - 48N_1(L) \cdot \kappa_{[L;1]} - 1152N_1(L) \cdot Z(L) \\
+ \frac{1}{2} N_0(L)\langle L, L\rangle_\Lambda^2 \cdot \kappa_{[L;1]} \ \in \ \mathsf{NL}^1(\mathcal{M}_\Lambda) \,.
\end{multline*}
After combining terms, we find
\begin{equation} \label{get3}
\Big(528N_1(L) + \frac{1}{2}N_0(L)\langle L, L\rangle_\Lambda^2\Big) \cdot \kappa_{[L;1]} + 6336N_1(L) \cdot Z(L) \ \in \ \mathsf{NL}^1(\mathcal{M}_\Lambda) \,.
\end{equation}

\vspace{8pt}
The system of equations \eqref{wd1}, \eqref{get1}, \eqref{get2}, and \eqref{get3} yields the
matrix
\begin{equation} \label{matx}
\left( \begin{array}{ccc}
-2 & 2\langle L, L\rangle_\Lambda & 0 \\
24N_1(L) & \frac{1}{2}N_0(L)\langle L, L\rangle_\Lambda^3 & 36N_1(L)\langle L, L\rangle_\Lambda \\
288N_1(L) & -12N_1(L)\langle L, L\rangle_\Lambda + \frac{1}{2}N_0(L)\langle L, L\rangle_\Lambda^3 & 288N_1(L)\langle L, L\rangle_\Lambda \\
0 & 528N_1(L) + \frac{1}{2}N_0(L)\langle L, L\rangle_\Lambda^2 & 6336N_1(L) \end{array} \right)\, .
\end{equation}
Since $N_0(L), N_1(L) \neq 0$, straightforward linear algebra{\footnote{One may even consider $\lambda$ as a $4^{\text{th}}$ variable in the equations \eqref{wd1}, \eqref{get1}, \eqref{get2}, and \eqref{get3}. For $\Lambda = (2\ell)$ and $L = H$, the only $\lambda$ terms are obtained from the unsplit contribution of Stratum 4 to ($\ddag$). We find the matrix
\begin{equation*}
\left( \begin{array}{cccc}
-2 & 2(2\ell) & 0 & 0\\
24N_1(\ell) & \frac{1}{2}N_0(\ell)(2\ell)^3 & 36N_1(\ell)(2\ell) & N_0(\ell)(2\ell)^3 \\
288N_1(\ell) & -12N_1(\ell)(2\ell) + \frac{1}{2}N_0(\ell)(2\ell)^3 & 288N_1(\ell)(2\ell) & N_0(\ell)(2\ell)^3 \\
0 & 528N_1(\ell) + \frac{1}{2}N_0(\ell)(2\ell)^2 & 6336N_1(\ell) & N_0(\ell)(2\ell)^2 \end{array} \right)
\end{equation*}
whose determinant is easily seen to be nonzero. In particular, we obtain a geometric proof of the fact
$$\lambda \ \in \ \mathsf{NL}^1(\mathcal{M}_{2\ell}) \,.$$
The determinant of the $4 \times 4$ matrix is likely nonzero for every $\Lambda$ and $H$ (in which case additional $\lambda$ terms appear). We plan to carry out more detailed computation in the future.}}
shows the matrix \eqref{matx} to have maximal 
rank 3.
We have therefore proven 
$$\kappa_{[L^3;0]}, \ \kappa_{[L;1]}, \ Z(L) \ \in \ \mathsf{NL}^1(\mathcal{M}_\Lambda) \,$$
and completed the analysis of Case A.

\vspace{8pt}
\noindent {\bf Case B.} $\kappa_{[H^2,L;0]}$ for $\langle L, L\rangle_\Lambda > 0$. \nopagebreak


\vspace{8pt}
We apply ($\ddag'$) with insertion $\mathcal{L}_{(1)}\mathcal{L}_{(2)}\mathcal{L}_{(3)}\in\mathsf{R}^3(\mathcal{X}^3_\Lambda)$, and push-down via $\pi^3_\Lambda$ 
to $\mathcal{M}_\Lambda$.
Since $$\kappa_{[H;1]}\, ,\ Z(H) \ \in \ \mathsf{NL}^1(\mathcal{M}_\Lambda)$$ by Case A, we find
$$2\ell \cdot \kappa_{[L^3;0]} - 3\langle L, L\rangle_\Lambda \cdot \kappa_{[H^2,L;0]} \ \in \ \mathsf{NL}^1(\mathcal{M}_\Lambda) \,.$$
Since $\kappa_{[L^3;0]} \in \mathsf{NL}^1(\mathcal{M}_\Lambda)$ by Case A, we have
$$\kappa_{[H^2,L;0]} \ \in \ \mathsf{NL}^1(\mathcal{M}_\Lambda) \,.$$
Case B is complete.

\vspace{8pt}
\noindent {\bf Case C.} $\kappa_{[L^3;0]}$,
$\kappa_{[H^2,L;0]}$, and $\kappa_{[L;1]}$ for $\langle L, L\rangle_\Lambda < 0$. \nopagebreak

\vspace{8pt}
\noindent \makebox[12pt][l]{$\bullet$}We apply ($\ddag'$) with insertion $\mathcal{L}_{(1)}\mathcal{L}_{(2)}\mathcal{L}_{(3)}\in\mathsf{R}^3(\mathcal{X}^3_\Lambda)$, and push-down
via $\pi^3_\Lambda$ to $\mathcal{M}_\Lambda$.
Since $$\kappa_{[H;1]}\, ,\ Z(H) \ \in \ \mathsf{NL}^1(\mathcal{M}_\Lambda)$$
by Case A, we find
\begin{equation} \label{get4}
2\ell \cdot \kappa_{[L^3;0]} - 3\langle L, L\rangle_\Lambda \cdot \kappa_{[H^2,L;0]} \ \in \ \mathsf{NL}^1(\mathcal{M}_\Lambda) \,.
\end{equation}

\vspace{8pt}
\noindent \makebox[12pt][l]{$\bullet$}We apply ($\ddag'$) with insertion $\mathcal{H}_{(1)}\mathcal{L}_{(2)}\mathcal{L}_{(3)}\in\mathsf{R}^3(\mathcal{X}^3_\Lambda)$,
and push-down via $\pi^3_\Lambda$ to $\mathcal{M}_\Lambda$.
Since $\kappa_{[H^3;0]} \in \mathsf{NL}^1(\mathcal{M}_\Lambda)$ by Case A, 
we find
\begin{equation} \label{get5}
2\ell \cdot \kappa_{[H,L^2;0]} - 2 \langle H, L\rangle_\Lambda \cdot \kappa_{[H^2,L;0]} \ \in \ \mathsf{NL}^1(\mathcal{M}_\Lambda) \,.
\end{equation}

\vspace{8pt}
\noindent \makebox[12pt][l]{$\bullet$}We apply ($\dag$) with respect to $L$, 
insert $\mathcal{H}_{(1)}\mathcal{H}_{(2)}\mathcal{L}_{(3)}
\mathcal{L}_{(4)}\in\mathsf{R}^4(\mathcal{X}^4_\Lambda)$, and
push-down via~$\pi^4_\Lambda$ to $\mathcal{M}_\Lambda$. We find
\begin{equation} \label{wd2}
\langle H, L\rangle^2 \cdot \kappa_{[L^3;0]} + \langle L, L\rangle^2 \cdot \kappa_{[H^2,L;0]} - 2\langle H, L\rangle \langle L, L\rangle \cdot \kappa_{[H,L^2;0]} \ \in \ \mathsf{NL}^1(\mathcal{M}_\Lambda) \,.
\end{equation}

\vspace{8pt}
\noindent \makebox[12pt][l]{$\bullet$}We apply ($\dag$) with respect to $L$, insert $\Delta_{(12)}\Delta_{(34)}\in\mathsf{R}^4(\mathcal{X}^4_\Lambda)$, and push-down via $\pi^4_\Lambda$ to~$\mathcal{M}_\Lambda$. We find
\begin{equation} \label{wd3}
2\langle L, L\rangle_\Lambda \cdot \kappa_{[L;1]} - 2 \cdot \kappa_{[L^3;0]} \ \in \ \mathsf{NL}^1(\mathcal{M}_\Lambda) \,.
\end{equation}

\vspace{8pt}
The system of equations \eqref{get4}, \eqref{get5}, and \eqref{wd2} 
for 
$$\kappa_{[L^3;0]}\, , \ \ \kappa_{[H,L^2;0]}\, , \ \
\kappa_{[H^2,L;0]}$$
yields
the matrix
$$\left( \begin{array}{ccc}
2\ell & 0 & -3\langle L, L\rangle_\Lambda \\
0 & 2\ell & -2\langle H, L\rangle_\Lambda \\
\langle H, L\rangle_\Lambda^2 & -2\langle H, L\rangle_\Lambda \langle L, L\rangle_\Lambda & \langle L, L\rangle_\Lambda^2
\end{array} \right)$$
with determinant 
$$2\ell \langle L, L\rangle_\Lambda \Big(2\ell \langle L, L\rangle_\Lambda - \langle H, L\rangle_\Lambda^2\Big) > 0 \, $$
by the Hodge index theorem applied to the second factor.
Therefore,
$$\kappa_{[L^3;0]} \,, \ \kappa_{[H,L^2;0]} \,, \ \kappa_{[H^2,L;0]} \ \in \ \mathsf{NL}^1(\mathcal{M}_\Lambda)\, ,$$
and by \eqref{wd3}, we have
$\kappa_{[L;1]} \in \mathsf{NL}^1(\mathcal{M}_\Lambda)$.
Case C is complete.

\vspace{8pt}
\noindent {\bf Case D.} $\kappa_{[L^3;0]}$, $\kappa_{[H^2,L;0]}$, $\kappa_{[L;1]}$, and $Z(L)$ for $\langle L, L\rangle_\Lambda = 0$. \nopagebreak

\vspace{8pt}
\noindent \makebox[12pt][l]{$\bullet$}We apply ($\ddag'$) with insertion $\mathcal{L}_{(1)}\mathcal{L}_{(2)}\mathcal{L}_{(3)}
\in\mathsf{R}^3(\mathcal{X}^3_\Lambda)$,
and push-down via $\pi^3_\Lambda$ to $\mathcal{M}_\Lambda$.
Since 
$$\kappa_{[H;1]}\, ,\ Z(H) \ \in \ \mathsf{NL}^1(\mathcal{M}_\Lambda)$$
by Case A, we find
$$2\ell \cdot \kappa_{[L^3;0]} - 3\langle L, L\rangle_\Lambda \cdot \kappa_{[H^2,L;0]} \ \in \ \mathsf{NL}^1(\mathcal{M}_\Lambda) \,,$$
hence{\footnote{A direct argument using
elliptically fibered $K3$ surfaces  shows $\kappa_{[L^3;0]} = 0$ for $\langle L, L\rangle_\Lambda = 0$.}}
$\kappa_{[L^3;0]} \in \mathsf{NL}^1(\mathcal{M}_\Lambda)$.

\vspace{8pt}
\noindent \makebox[12pt][l]{$\bullet$}We apply ($\ddag'$) with insertion $\mathcal{H}_{(1)}\mathcal{L}_{(2)}\mathcal{L}_{(3)}\in\mathsf{R}^3(\mathcal{X}^3_\Lambda)$,
and push-down via $\pi^3_\Lambda$ to $\mathcal{M}_\Lambda$.
We find
\begin{equation} \label{get6}
2\ell \cdot \kappa_{[H,L^2;0]} - 2 \langle H, L\rangle_\Lambda \cdot \kappa_{[H^2,L;0]} \ \in \ \mathsf{NL}^1(\mathcal{M}_\Lambda) \,.
\end{equation}

\vspace{8pt}
\noindent \makebox[12pt][l]{$\bullet$}We apply ($\dag$) with respect to $L$, 
insert $\mathcal{H}_{(1)}\mathcal{H}_{(2)}\Delta_{(34)}
\in\mathsf{R}^4(\mathcal{X}^4_\Lambda)$,
and push-down via~$\pi^4_\Lambda$ to $\mathcal{M}_\Lambda$.
We find
$$\langle H, L\rangle_\Lambda^2 \cdot \kappa_{[L;1]} - 2\langle H, L\rangle_\Lambda \cdot \kappa_{[H,L^2;0]} \ \in \ \mathsf{NL}^1(\mathcal{M}_\Lambda) \,.$$
Since $\langle H, L\rangle_\Lambda \neq 0$ by the Hodge index theorem, we have
\begin{equation} \label{wd4}
\langle H, L\rangle_\Lambda \cdot \kappa_{[L;1]} - 2 \cdot \kappa_{[H,L^2;0]} \ \in \ \mathsf{NL}^1(\mathcal{M}_\Lambda) \,.
\end{equation}

\vspace{8pt}
\noindent \makebox[12pt][l]{$\bullet$}We apply ($\ddag$) with respect to $L$,
insert $\mathcal{H}_{(1)}\mathcal{H}_{(2)}\mathcal{H}_{(3)}\mathcal{L}_{(4)}
\in\mathsf{R}^4(\mathcal{X}^4_\Lambda)$,
and push-down via~$\pi^4_\Lambda$ to $\mathcal{M}_\Lambda$.
We find
\begin{multline*}
36N_1(L) \langle H, L\rangle_\Lambda \cdot \kappa_{[H^2,L;0]} + 36N_1(L)(2\ell) \cdot \kappa_{[H,L^2;0]} + 36N_1(L)(2\ell) \langle H, L\rangle_\Lambda \cdot Z(L) \\
- 36N_1(L) \langle H, L\rangle_\Lambda \cdot \kappa_{[H^2,L;0]} \ \in \ \mathsf{NL}^1(\mathcal{M}_\Lambda) \,.
\end{multline*}
Since $N_1(L) \neq 0$, we have
\begin{equation} \label{get7}
\kappa_{[H,L^2;0]} + \langle H, L\rangle_\Lambda \cdot Z(L) \ \in \ \mathsf{NL}^1(\mathcal{M}_\Lambda) \,.
\end{equation}

\vspace{8pt}
\noindent \makebox[12pt][l]{$\bullet$}We apply ($\ddag$) with respect to $L$,
insert $\mathcal{H}_{(1)}\mathcal{H}_{(2)}\Delta_{(34)}
\in\mathsf{R}^4(\mathcal{X}^4_\Lambda)$,
and push-down via~$\pi^4_\Lambda$ to $\mathcal{M}_\Lambda$. We find
\begin{multline*}
288N_1(L) \cdot \kappa_{[H^2,L;0]} + 12N_1(L)(2\ell) \cdot \kappa_{[L;1]} + 48N_1(L) \cdot \kappa_{[H^2,L;0]} \\
+ 288N_1(L)(2\ell) \cdot Z(L) + 24N_1(L)(2\ell) \cdot Z(L) 
\\- 24N_1(L) \cdot \kappa_{[H^2,L;0]} 
- 24N_1(L) \cdot \kappa_{[H^2,L;0]} - 24N_1(L)(2\ell) \cdot Z(L) \ \in \ \mathsf{NL}^1(\mathcal{M}_\Lambda) \,.
\end{multline*}
After combining terms, we obtain
\begin{equation} \label{get8}
24 \cdot \kappa_{[H^2,L;0]} + 2\ell \cdot \kappa_{[L;1]} + 24(2\ell) \cdot Z(L) \ \in \ \mathsf{NL}^1(\mathcal{M}_\Lambda) \,.
\end{equation}

\vspace{8pt}
We multiply \eqref{get8} by $\langle H, L\rangle_\Lambda$, and make substitutions using \eqref{get6}, \eqref{wd4}, and \eqref{get7}, which yields
$$(12 + 2 - 24)(2\ell) \cdot \kappa_{[H,L^2;0]} \ \in \ \mathsf{NL}^1(\mathcal{M}_\Lambda) \,.$$
Therefore,
$\kappa_{[H,L^2;0]} \in \mathsf{NL}^1(\mathcal{M}_\Lambda)$.
Then, again by \eqref{get6}, \eqref{wd4}, and \eqref{get7},
$$\kappa_{[H^2,L;0]}\,, \ \kappa_{[L;1]}\,, \ Z(L) \ \in \ \mathsf{NL}^1(\mathcal{M}_\Lambda) \,.$$
Case D is complete.

\vspace{8pt}
\noindent {\bf Case E.} $\kappa_{[L_1,L_2,L_3;0]}$ for arbitrary $L_1, L_2, L_3 \in \Lambda$. \nopagebreak

\vspace{8pt}
We apply ($\ddag'$) with insertion $\mathcal{L}_{1,(1)}\mathcal{L}_{2,(2)}\mathcal{L}_{3,(3)}
\in\mathsf{R}^3(\mathcal{X}^3_\Lambda)$,
and push-down via $\pi^3_\Lambda$ to~$\mathcal{M}_\Lambda$.
The result
expresses $2\ell \cdot \kappa_{[L_1,L_2,L_3;0]}$ in terms of Noether-Lefschetz divisors and $\kappa$ divisors treated in the previous cases. Therefore,
$$\kappa_{[L_1,L_2,L_3;0]} \ \in \ \mathsf{NL}^1(\mathcal{M}_\Lambda) \,.$$
Case E is complete.

\vspace{8pt}
Cases A-E together cover all divisorial $\kappa$ classes and
prove the divisorial case of Theorem \ref{dxxd}.

\begin{proposition} \label{pdxxd} The strict tautological ring 
in codimension $1$ is generated
by Noether-Lefschetz loci,
$$\mathsf{NL}^1(\mathcal{M}_{\Lambda}) = \mathsf{R}^1(\mathcal{M}_{\Lambda})\, .$$
\end{proposition}

In fact, by the result of \cite{Ber}, $\mathsf{NL}^1(\mathcal{M}_\Lambda)$
generates {\it all} of $\mathsf{A}^1(\mathcal{M}_\Lambda)$
for $\text{rank}(\Lambda)\leq 17$. 
We have given a direct proof
of Proposition \ref{pdxxd}
using exported relations which is valid for every lattice
polarization $\Lambda$ without rank restriction. 
The same method will be used to prove the full statement of
Theorem \ref{dxxd}.

\subsection{Second Chern class} \label{scc}
The next step is to eliminate the $c_2(\mathcal{T}_{\pi_\Lambda})$ index 
in the class $\kappa_{[L_1^{a_1},\ldots,L_k^{a_k};b]}$ and reduce to the case
$$\kappa_{[L_1^{a_1},\ldots,L_k^{a_k};0]} \,.$$
Our strategy is to express $c_2(\mathcal{T}_{\pi_\Lambda}) \in \mathsf{R}^2(\mathcal{X}_\Lambda)$ in terms of simpler strict tautological classes.

From now on, we will require only the decomposition $(\ddag')$. 

\vspace{8pt}
\noindent \makebox[12pt][l]{$\bullet$}We apply ($\ddag'$) with insertion $\mathcal{H}_{(1)}\mathcal{H}_{(2)}\Delta_{(23)}
\in\mathsf{R}^4(\mathcal{X}^3_\Lambda)$, and push-down via $\pi^3_\Lambda$ to $\mathcal{M}_\Lambda$.
As a result, we find
$$2\ell \cdot \kappa_{[H^2;1]} - \kappa_{[H^3;0]}\kappa_{[H;1]} - 2 \cdot \kappa_{[H^4;0]} + 2 \cdot \kappa_{[H^4;0]} \ \in \ \mathsf{NL}^2(\mathcal{M}_\Lambda) \,,$$
where we have used Proposition \ref{pdxxd} for all the
non-principal terms corresponding to larger lattices.
By Proposition \ref{pdxxd} for $\Lambda$, we have
$\kappa_{[H^3;0]}, \,\kappa_{[H;1]} \in 
\mathsf{NL}^1(\mathcal{M}_{\Lambda})$. 
We conclude
$$\kappa_{[H^2;1]} \ \in \ \mathsf{NL}^2(\mathcal{M}_\Lambda) \,.$$

\vspace{8pt}
\noindent \makebox[12pt][l]{$\bullet$}We apply ($\ddag'$) with insertion $\Delta_{(12)}\in\mathsf{R}^2(\mathcal{X}^3_\Lambda)$, and push-forward to $\mathcal{X}_\Lambda$ via the third projection
$$\text{pr}_{(3)}: \mathcal{X}^3_\Lambda \rightarrow \mathcal{X}_\Lambda\, .$$
We find
\begin{align*}
2\ell \cdot c_2(\mathcal{T}_{\pi_\Lambda}) \, & = \, 2 \cdot \mathcal{H}^2 + 24 \cdot \mathcal{H}^2 - \kappa_{[H^2;1]} - 2 \cdot \mathcal{H}^2 + \ldots \\
& = \, 24 \cdot \mathcal{H}^2 - \kappa_{[H^2;1]} + \ldots \ \in \ \mathsf{R}^2(\mathcal{X}_\Lambda) \,,
\end{align*}
where the dots stand for strict tautological classes supported over proper Noether-Lefschetz loci of $\mathcal{M}_\Lambda$.

\vspace{8pt}
We have already proven $\kappa_{[H^2;1]} \in \mathsf{NL}^2(\mathcal{M}_\Lambda)$. Therefore, up to strict tautological classes supported over proper Noether-Lefschetz loci of $\mathcal{M}_\Lambda$, we may replace $c_2(\mathcal{T}_{\pi_\Lambda})$ by
$$\frac{24}{2\ell} \,\cdot\, \mathcal{H}^2 \ \in \ \mathsf{R}^2(\mathcal{X}_\Lambda) \,.$$
The replacement lowers the $c_2(\mathcal{T}_{\pi_\Lambda})$ index of  $\kappa$ classes.
By induction,
we need only prove Theorem \ref{dxxd} for $\kappa$ classes with trivial $c_2(\mathcal{T}_{\pi_\Lambda})$ index.

\subsection{Proof of Theorem \ref{dxxd}}
The $\kappa$ classes with trivial $c_2(\mathcal{T}_{\pi_\Lambda})$ index can be written as
$$\kappa_{[H^a,L_1,\ldots,L_k;0]} \ \in \ \mathsf{R}^{a + k - 2}(\mathcal{M}_\Lambda) \,,$$
where the $L_i \in \Lambda$ are admissible classes 
(not necessarily distinct) that are different from the
quasi-polarization $H$. 

\vspace{8pt}
\noindent {\bf Codimension $2$.} \nopagebreak

\vspace{8pt}
In codimension $2$, the complete list of
$\kappa$ classes (with trivial $c_2(\mathcal{T}_{\pi_\Lambda})$ index)
is:
$$\kappa_{[H^4;0]}\,, \ \kappa_{[H^3,L;0]}\,, \ \kappa_{[H^2,L_1,L_2;0]}\,, \ \kappa_{[H,L_1,L_2,L_3;0]}\,, \ \kappa_{[L_1,L_2,L_3,L_4;0]} \ \in \ \mathsf{R}^2(\mathcal{M}_\Lambda) \,.$$

\vspace{8pt}
\noindent \makebox[12pt][l]{$\bullet$}For $\kappa_{[H^4;0]}$, we apply ($\ddag'$) with insertion $\mathcal{H}_{(1)}^2\Delta_{(23)}\in\mathsf{R}^4(\mathcal{X}^3_\Lambda)$, and push-down via $\pi^3_\Lambda$ to $\mathcal{M}_\Lambda$. 
We find
$$2\ell \cdot \kappa_{[H^2;1]} - 24 \cdot \kappa_{[H^4;0]} - 2 \cdot \kappa_{[H^4;0]} + 2 \cdot \kappa_{[H^4;0]} + 2\ell \cdot \kappa_{[H^2;1]} \ \in \ \mathsf{NL}^2(\mathcal{M}_\Lambda) \,,$$
where we have used Proposition \ref{pdxxd} for all the
non-principal terms corresponding to larger lattices.
Since $\kappa_{[H^2;1]} \in \mathsf{NL}^2(\mathcal{M}_\Lambda)$
by Section \ref{scc}, we have
$\kappa_{[H^4;0]} \in \mathsf{NL}^2(\mathcal{M}_\Lambda)$.

\vspace{8pt}
\noindent \makebox[12pt][l]{$\bullet$}For $\kappa_{[H^3,L;0]}$, we apply ($\ddag'$) with insertion $\mathcal{H}_{(1)}^2\mathcal{H}_{(2)}\mathcal{L}_{(3)}
\in\mathsf{R}^4(\mathcal{X}^3_\Lambda)$, and push-down via~$\pi^3_\Lambda$ to $\mathcal{M}_\Lambda$. We find
$$2\ell \cdot \kappa_{[H^3,L;0]} - \langle H, L\rangle_\Lambda \cdot \kappa_{[H^4;0]} - 2 \cdot \kappa_{[H^3;0]} \kappa_{[H^2,L;0]} + 2\ell \cdot \kappa_{[H^3,L;0]} \ \in \ \mathsf{NL}^2(\mathcal{M}_\Lambda) \,,$$
hence
$\kappa_{[H^3,L;0]} \in \mathsf{NL}^2(\mathcal{M}_\Lambda)$.

\vspace{8pt}
\noindent \makebox[12pt][l]{$\bullet$}For $\kappa_{[H^2,L_1,L_2;0]}$, we apply ($\ddag'$) with insertion $\mathcal{H}_{(1)}^2\mathcal{L}_{1,(2)}\mathcal{L}_{2,(3)} \in\mathsf{R}^4(\mathcal{X}^3_\Lambda)$, and push-down via $\pi^3_\Lambda$ to $\mathcal{M}_\Lambda$. We find
\begin{multline*}
2\ell \cdot \kappa_{[H^2,L_1,L_2;0]} - \langle L_1, L_2\rangle_\Lambda \cdot \kappa_{[H^4;0]} \\
- 2 \cdot \kappa_{[H^2,L_1;0]}\kappa_{[H^2,L_2;0]} + 2\ell \cdot \kappa_{[H^2,L_1,L_2;0]} \ \in \ \mathsf{NL}^2(\mathcal{M}_\Lambda) \,,
\end{multline*}
hence
$\kappa_{[H^2,L_1,L_2;0]} \in \mathsf{NL}^2(\mathcal{M}_\Lambda)$.

\vspace{8pt}
\noindent \makebox[12pt][l]{$\bullet$}For $\kappa_{[H,L_1,L_2,L_3;0]}$, we apply ($\ddag'$) with insertion $\mathcal{H}_{(1)}\mathcal{L}_{1,(1)}\mathcal{L}_{2,(2)}\mathcal{L}_{3,(3)}
\in \mathsf{R}^4(\mathcal{X}^3_\Lambda)$, and push-down via $\pi^3_\Lambda$ to $\mathcal{M}_\Lambda$. We find
\begin{multline*}
2\ell \cdot \kappa_{[H,L_1,L_2,L_3;0]} - \langle L_2, L_3\rangle_\Lambda \cdot \kappa_{[H^3,L_1;0]} - \kappa_{[H^2,L_2;0]}\kappa_{[H,L_1,L_3;0]} \\
- \kappa_{[H^2,L_3;0]}\kappa_{[H,L_1,L_2;0]} + \langle H, L_1\rangle_\Lambda \cdot \kappa_{[H^2,L_2,L_3;0]} \ \in \ \mathsf{NL}^2(\mathcal{M}_\Lambda) \,,
\end{multline*}
hence
$\kappa_{[H,L_1,L_2,L_3;0]} \in \mathsf{NL}^2(\mathcal{M}_\Lambda)$.

\vspace{8pt}
\noindent \makebox[12pt][l]{$\bullet$}For $\kappa_{[L_1,L_2,L_3,L_4;0]}$, we apply ($\ddag'$) with insertion $\mathcal{L}_{1,(1)}\mathcal{L}_{2,(1)}\mathcal{L}_{3,(2)}\mathcal{L}_{4,(3)}
\in \mathsf{R}^4(\mathcal{X}^3_\Lambda)$, and push-down via $\pi^3_\Lambda$ to $\mathcal{M}_\Lambda$. We find
\begin{multline*}
2\ell \cdot \kappa_{[L_1,L_2,L_3,L_4;0]} - \langle L_3, L_4\rangle_\Lambda \cdot \kappa_{[H^2,L_1,L_2;0]} - \kappa_{[H^2,L_3;0]}\kappa_{[L_1,L_2,L_4;0]} \\
- \kappa_{[H^2,L_4;0]}\kappa_{[L_1,L_2,L_3;0]} + \langle L_1, L_2\rangle_\Lambda \cdot \kappa_{[H^2,L_3,L_4;0]} \ \in \ \mathsf{NL}^2(\mathcal{M}_\Lambda) \,,
\end{multline*}
hence
$\kappa_{[L_1,L_2,L_3,L_4;0]} \in \mathsf{NL}^2(\mathcal{M}_\Lambda)$.

\vspace{8pt}
\noindent {\bf Codimension $\geq 3$.} \nopagebreak

\vspace{8pt}
Our strategy in codimension $c\geq 3$ involves an induction on codimension together with a second induction on the $H$ index $a$ of the kappa class 
$$\kappa_{[H^a,L_1,\ldots,L_k;0]} \ \in \ \mathsf{R}^{a + k - 2}(\mathcal{M}_\Lambda) \,.$$
For the induction on $c$, we assume the Noether-Lefschetz generation
for all {\it lower} codimension. The base case is Proposition \ref{pdxxd}.
For the induction on $a$, we assume the Noether-Lefschetz generation
for all {\it higher} $H$ index.

\vspace{8pt}
\noindent \makebox[12pt][l]{$\bullet$}For the base of the induction on $H$ index, consider
the class 
$$\kappa_{[H^a;0]} \ \in \ \mathsf{R}^{a - 2}(\mathcal{M}_\Lambda)\, .$$ 
We apply ($\ddag'$), insert
$$\mathcal{H}_{(1)}^{a - 3}\mathcal{H}_{(2)}^{2}\mathcal{H}_{(3)} \ \in \ \mathsf{R}^a(\mathcal{X}^3_\Lambda)\, \ \ \text{with}  \ a-2=c\, ,$$
and push-down via $\pi^3_\Lambda$ to $\mathcal{M}_\Lambda$.
By the induction on codimension, we obtain
\begin{multline} \label{fnfn}
2\ell \cdot \kappa_{[H^a;0]} - 2 \cdot \kappa_{[H^3;0]}\kappa_{[H^{a-1};0]} - \kappa_{[H^4;0]}\kappa_{[H^{a-2};0]} \\
+ 2\ell \cdot \kappa_{[H^a;0]} + \kappa_{[H^5;0]}\kappa_{[H^{a-3};0]} \ \in \ \mathsf{NL}^{a - 2}(\mathcal{M}_\Lambda) \,.
\end{multline}
For both{\footnote{Since $a-2=c\geq 3$, $a\geq 5$.}} $a = 5$ and $a > 5$, the coefficient of $\kappa_{[H^a;0]}$ is positive and
the other terms in \eqref{fnfn} are products of $\kappa$ classes of lower codimension. Therefore, by the induction hypothesis,
$$\kappa_{[H^a;0]} \ \in \ \mathsf{NL}^{a - 2}(\mathcal{M}_\Lambda) \,.$$

\vspace{8pt}
\noindent \makebox[12pt][l]{$\bullet$}If $a>0$ and $k > 0$, we apply ($\ddag'$), insert
$$\mathcal{H}_{(1)}^{a - 1}\mathcal{L}_{1,(1)} \cdots \mathcal{L}_{k - 1,(1)}\mathcal{H}_{(2)}\mathcal{L}_{k,(3)} \ \in \ \mathsf{R}^{a+k}(\mathcal{X}^3_\Lambda)\, \ \ \text{with}  \ a+k-2=c\, ,$$
and push-down via $\pi^3_\Lambda$ to $\mathcal{M}_\Lambda$.
By the induction on codimension, we obtain
\begin{multline}\label{msms7}
2\ell \cdot \kappa_{[H^a,L_1,\ldots,L_k;0]} - \langle H, L_k\rangle_\Lambda \cdot \kappa_{[H^{a + 1},L_1,\ldots,L_{k-1};0]} \\
- \kappa_{[H^3;0]}\kappa_{[H^{a-1},L_1,\ldots,L_{k-1},L_k;0]} - \kappa_{[H^2,L_k;0]}\kappa_{[H^a,L_1,\ldots,L_{k-1};0]} \\
+ \kappa_{[H^3,L_k;0]}\kappa_{[H^{a-1},L_1,\ldots,L_{k-1};0]} \ \in \ \mathsf{NL}^{a + k - 2}(\mathcal{M}_\Lambda) \,.
\end{multline}
Since the last three terms of \eqref{msms7} are products of $\kappa$ classes of lower codimension (since $a+k\geq 5$), using the induction hypothesis again yields
$$2\ell \cdot \kappa_{[H^a,L_1,\ldots,L_k;0]} - \langle H, L_k\rangle_\Lambda \cdot \kappa_{[H^{a + 1},L_1,\ldots,L_{k-1};0]} \ \in \ \mathsf{NL}^{a + k - 2}(\mathcal{M}_\Lambda) \,,$$
which allows us to raise the $H$ index.

\vspace{8pt}
\noindent \makebox[12pt][l]{$\bullet$}If $a = 0$, we apply ($\ddag'$), insert
$$\mathcal{L}_{1,(1)} \cdots \mathcal{L}_{k - 2,(1)}\mathcal{L}_{k - 1, (2)}\mathcal{L}_{k,(3)} \ \in \ \mathsf{R}^k(\mathcal{X}^3_\Lambda)\, \ \ \text{with}  \ k-2=c\, ,$$
and push-down via $\pi^3_\Lambda$ to $\mathcal{M}_\Lambda$.
By the induction on codimension, we obtain
\begin{multline} \label{msms}
2\ell \cdot \kappa_{[L_1,\ldots,L_k;0]} - \langle L_{k - 1}, L_k\rangle_\Lambda \cdot \kappa_{[H^2,L_1,\ldots,L_{k - 2};0]} \\
- \kappa_{[H^2,L_{k - 1};0]}\kappa_{[L_1,\ldots,L_{k - 2},L_k;0]} - \kappa_{[H^2,L_k;0]}\kappa_{[L_1,\ldots,L_{k - 2},L_{k - 1};0]} \\
+ \kappa_{[H^2,L_{k-1},L_k;0]}\kappa_{[L_1,\ldots,L_{k-2};0]} \ \in \ \mathsf{NL}^{k - 2}(\mathcal{M}_\Lambda) \,.
\end{multline}
Since the last three terms of \eqref{msms} are products of $\kappa$ classes of lower codimension (since $k\geq 5$), using the induction hypothesis again yields
$$2\ell \cdot \kappa_{[L_1,\ldots,L_k;0]} - \langle L_{k - 1}, L_k\rangle_\Lambda \cdot \kappa_{[H^2,L_1,\ldots,L_{k - 2};0]} \ \in \ \mathsf{NL}^{k - 2}(\mathcal{M}_\Lambda) \,,$$
which allows us to raise the $H$ index.

\vspace{8pt}
The induction argument on codimension and $H$ index is complete. The Noether-Lefschetz generation of Theorem \ref{dxxd} is proven. \qed

\end{document}